\documentclass[a4paper,leqno]{amsart}

\usepackage{amsmath}
\usepackage{amsthm}
\usepackage{amssymb}
\usepackage{amsfonts}
\usepackage[applemac]{inputenc}
\usepackage{mathrsfs}
\usepackage{color}
\usepackage{pdfsync}

\newcommand{\A}{\mathcal{A}}

\newcommand{\dw}{\mathrm{d}_{L^2}}
\newcommand{\ds}{\mathrm{d}_\Sigma}
\newcommand{\Ec}{\mathcal{E}}

\newcommand{\db}{\mathrm{d}_{\bullet}}
\newcommand{\dr}{\mathrm{d}}

\newcommand{\wb}{{\omega}_{\bullet}}

\newcommand{\dd}{\,d}

\def\R{{\mathbb R}}
\def\N{{\mathbb N}}

\def\le{\leqslant}
\def\ge{\geqslant}

\newcommand{\re}{\mathrm{Re}}
\newcommand{\im}{\mathrm{Im}}

\theoremstyle{plain}
\newtheorem{theorem}{Theorem}[section]
\newtheorem{lemma}[theorem]{Lemma}
\newtheorem{corollary}[theorem]{Corollary}
\newtheorem{proposition}[theorem]{Proposition}

\theoremstyle{definition}
\newtheorem{definition}[theorem]{Definition}

\newtheorem{remark}[theorem]{Remark}
\newtheorem*{remark*}{Remark}

\numberwithin{equation}{section}

\begin{document}


\title[Global attractor in a model for rotating BEC]
{Global attractor for a Ginzburg-Landau type model of rotating Bose-Einstein condensates}

\author[A.\ Cheskidov]{Alexey Cheskidov}

\address[A.\ Cheskidov]{Department of Mathematics, Statistics, and Computer Science\\
University of Illinois at Chicago\\
322 Science and Engineering Offices (M/C 249)\\
851 South Morgan Street\\
Chicago, Illinois 60607\\ United States}
\email{acheskid@math.uic.edu}

\author[D.\ Marahrens]{Daniel Marahrens}

\address[D.\ Marahrens]{Max-Planck-Institute for Mathematics in the Sciences\\
Inselstrasse 22, 04103 Leipzig\\ Germany}
\email{daniel.marahrens@mis.mpg.de}

\author[C.\ Sparber]{Christof Sparber}

\address[C.\ Sparber]{Department of Mathematics, Statistics, and Computer Science\\
University of Illinois at Chicago\\
322 Science and Engineering Offices (M/C 249)\\
851 South Morgan Street\\
Chicago, Illinois 60607\\ United States}
\email{sparber@uic.edu}

\begin{abstract}
We study the long time behavior of solutions to a nonlinear partial differential equation arising in the 
description of trapped rotating Bose-Einstein condensates. 
The equation can be seen as a hybrid between the well-known nonlinear Schr\"odinger/Gross-Pitaevskii equation and the 
Ginzburg-Landau equation. We prove existence and uniqueness of global in-time solutions in the physical energy space 
and establish the existence of a global attractor within the associated dynamics. 
We also obtain basic structural properties of the attractor and an estimate on its Hausdorff and fractal dimensions. 
As a by-product, we establish heat-kernel estimates on the linear part of the equation.
\end{abstract}

\date{\today}


\thanks{This publication is based on work supported by the NSF through grant nos. DMS-1161580 and DMS-1348092}
\maketitle


\section{Introduction}
\label{sec:intro}

\subsection{Physical motivation} The study of quantized vortex dynamics in Bose-Einstein condensates (BECs)
is a topic of intense experimental and theoretical investigations.  A particular interesting situation is created when the BEC is stirred through an
external {\it rotating} confinement potential. Indeed, if the rotation speed exceeds some critical value 
{\it vortices} and, more generally, {\it vortex lattices} are being created, see, e.g., \cite{A, CPRY} for a broader introduction. 

From a mathematical point of view, rotating BECs can be described within the realm of a mean-field model, the so-called 
{\it Gross-Pitaevskii equation} \cite{PiSt}. In the following, we shall assume, without loss of generality, that the system rotates around the $z$-axis with a 
given speed $\Omega \in \R$. Placing ourselves in the associated rotating reference frame, 
the corresponding mathematical model is a {\it nonlinear 
Schr\"odinger equation} (NLS) given by
\begin{equation}
\label{NLS_rot}
i \partial_t \psi = -\frac{1}{2} \Delta \psi + \lambda |\psi|^{2} \psi + V(x) \psi - \Omega L \psi .
\end{equation}
Here, $t\in \R$, $x\in \R^d$ with $d=3$, or $d=2$, respectively. The latter corresponds to the assumption of homogeneity of the BEC 
along the $z$-axis (see, e.g., \cite{BMSW, MeSp}, for a rigorous scaling limit from three to effective two-dimensional models for BEC). The parameter $\lambda \ge 0$ describes 
the strength of the inter-particle interaction, which in this work is assumed to be repulsive.
The potential $V$ describes the magnetic trap and is usually taken in the form of a harmonic oscillator, i.e.
\begin{equation}
\label{Vquadr}
V(x)=\frac{1}{2} \omega^2 |x|^2, \quad \omega \in \R.
\end{equation}
Here, and in the following, we choose $V$ to be rotationally symmetric for simplicity. All our results can be easily generalized to the case of 
an anisotropic harmonic oscillator. 
Finally, $ \Omega L\psi$ describes the rotation around the $z$-axis, where 
\begin{equation}\label{eq:angular_momentum}
 L \psi:= - i (x_1 \partial_{x_2} \psi - x_2 \partial_{x_1}\psi)\equiv -i x^\perp \cdot \nabla \psi,
\end{equation}
denotes the corresponding quantum mechanical rotation operator. 

Most rigorous mathematical results on vortex creation are based on standing wave solutions of \eqref{NLS_rot}, i.e. 
solutions of the form $\psi(t,x) = \varphi(x) e^{-i \mu t}$, $\mu\in \R$, which leads to the following nonlinear elliptic equation
\begin{equation}
\label{stat_NLS}
-\frac{1}{2} \Delta \varphi + \lambda |\varphi|^{2} \varphi + V(x) \varphi - \Omega L \varphi - \mu \varphi =0.
\end{equation}
Equation \eqref{stat_NLS} can be interpreted as the 
Euler-Lagrange equation of the associated Gross-Pitaevskii energy functional \cite{PiSt, S}:
\begin{equation}\label{GPenergy}
E_{\rm GP}(\varphi):= \int_{\R^d} \bigg( \frac{1}{2} |\nabla \varphi|^2 + V(x) |\psi|^2 + \frac{\lambda}{2}|\varphi |^{4} -  \Omega \overline{\varphi}
L \varphi \bigg) \;dx,
\end{equation}
One possible way of constructing solutions to \eqref{stat_NLS} is thus to minimize \eqref{GPenergy} under the constraint $\|\varphi \|^2_{L^2}=M$, 
where $M>0$ denotes a given mass. This consequently yields a {\it chemical potential} $\mu=\mu(M)\ge 0$ playing the role of a Lagrange multiplier. 
In order to do so, one requires $\omega>|\Omega| $ which ensures that $E_{\rm GP}$ is bounded from below. 
Physically speaking, this condition means that the confinement potential $V(x)$ is stronger than the rotational forces, 
ensuring that the BEC stays trapped.
Within this framework, it was proved in \cite{S} 
that the hereby obtained physical {\it ground states}, i.e. energy minimizing solutions of \eqref{stat_NLS}, undergo a symmetry breaking (of the rotational symmetry) 
for sufficiently strong $\Omega$ and/or $\lambda \ge 0$. The latter is interpreted as the onset of vortex-lattice creation. 

On the other hand, it is often argued in the physics literature that a small amount of dissipation must be present for the 
experimental realization of stable vortex lattices, cf. \cite{CMK, KT, KNKM}. 
In order to describe such dissipative effects, not present in the original Gross-Pitaevskii equation \eqref{NLS_rot}, 
the following phenomenological model has been proposed in \cite{TKU} and subsequently been studied in, e.g., \cite{CGKT, CP, GAF, KT, KNKM}:
\begin{equation}\label{dissGP}
(i\beta-\gamma) \partial_t \psi = -\frac{1}{2} \Delta \psi + \lambda |\psi|^{2} \psi + V(x) \psi - \Omega L \psi - \mu \psi.
\end{equation}
Here $\beta\in\R$ and $\gamma>0$ are physical parameters whose ratio describes the strength of the dissipation. (In \cite{GAF} the authors use 
formal arguments based on quantum kinetic theory to obtain $\frac{\gamma}{\beta} \approx
0.03$.)  Note that any {\it time-independent solution} $\psi = \varphi(x)$ of \eqref{dissGP} solves the stationary NLS \eqref{stat_NLS}. 
In contrast to \eqref{NLS_rot}, equation \eqref{dissGP} is no longer Hamiltonian and only makes sense for $t\in \R_+$.

\subsection{Mathematical setting and main result} 
This work is devoted to a rigorous mathematical analysis of \eqref{dissGP}. In particular, we shall be interested in the long time behavior of its solutions as $t \to +\infty$. 
To this end, it is convenient to re-scale time such that $\beta^2+\gamma^2=1$. 
Then we can write \[
i\beta-\gamma = -e^{i\vartheta}, \quad \text{for some $\vartheta\in\Big(-\frac \pi  2,\frac \pi  2\Big)$.}
\] 
Note that by doing so, the real part of $e^{i\vartheta}$ has the same (positive) sign as $\gamma > 0$. We shall thus be concerned with the following initial value problem for $(t,x)\in \R_+\times \R^d$ and $d=2,3$:
\begin{equation}
\label{NLS_diss}
-e^{i\vartheta}\partial_t \psi = -\frac{1}{2} \Delta \psi + \lambda |\psi|^{2\sigma} \psi + V(x) \psi - \Omega  L \psi - \mu \psi, \quad 
\psi _{\mid t =0}  =\psi_0(x),
\end{equation}
where $\psi_0$ will be chosen in some appropriate function space (see below), and $\sigma > 0$ a generalized nonlinearity. 
Formally, the usual Gross-Pitaevskii equation \eqref{NLS_rot} is obtained from \eqref{NLS_diss} in the limit $\vartheta \to \pm \frac{\pi}{2}$. On the other hand, if $\vartheta =0$ the Hamiltonian character 
of the model is completely lost and \eqref{NLS_diss} instead resembles a nonlinear parabolic 
equation of {\it complex Ginzburg-Landau} (GL) type, cf. \cite{AK} for a review on this type of models.

Equation \eqref{NLS_diss} can thus be seen as a hybrid between the Gross-Pitaevskii/Non-linear Schr\"odinger equation and 
the complex Ginzburg-Landau equation. Both kind of models have been extensively studied in the mathematical literature:
For local and global well-posedness results on NLS, 
with or without quadratic potentials $V$, we refer to \cite{Caze, Car05, Car09}. Allowing for the inclusion of a rotation term, the initial value problem for \eqref{NLS_rot} 
has been analyzed in \cite{AMS}.
Similarly, well-posedness results for the complex GL equation in various spaces can be found in \cite{DGL, GV1,GV2}. The existence and basic properties of 
a global attractor for solutions to GL (on bounded domains $D\subset \R^d$) are studied in \cite{Te} and \cite{MM}. 
Moreover, the so-called {\it inviscid limit} which links solutions of 
GL to solutions of NLS has been established in \cite{Wu}. However, none of the aforementioned results directly apply to the model 
\eqref{NLS_diss}, which involves an unbounded (quadratic) potential $V$ and a rotation term, 
neither of which have been included in the studies on GL cited above. 
One should also note that the GL equation in its most general form allows for 
different complex pre-factors in front of the Laplacian and the nonlinearity. In our case those
pre-factors coincide, allowing for a closer connection to NLS.
Very recently, a similar type of such restricted GL models with $\lambda <0$ (and without potential and rotation terms) has 
been studied in \cite{CDW1, CDW2} as an ``intermediate step"
between the NLS and the nonlinear heat equation. 
Finally, we also mention that equation \eqref{dissGP} with $\beta=0$ is used to 
numerically obtain the Gross-Pitaevskii ground states, cf. \cite{BaoDu, CST}.

As announced before, we shall mainly be interested in the long time behavior of solutions to \eqref{NLS_diss}. In view of this the main result 
of our paper can be stated in the following form:

\begin{theorem}\label{thmmain} Let $d\in \{2,3\}$, $\omega>|\Omega|$, $\vartheta \in (-\frac \pi  2 , \frac \pi  2 )$, $\lambda \ge  0$, and $0< \sigma < \frac{d}{2(d-2)}$. 
Then for any 
\[
\psi_0 \in \Sigma:=\big\{f\in H^1(\R^d)\; : \; |x|f\in L^2(\R^d)\big\}
\]
there exits a unique strong solution $\psi \in C([0,\infty), \Sigma)$ to \eqref{NLS_diss}.  The associated mass and energy thereby satisfy the identities \eqref{eq:mass} and \eqref{eq:energy} below. 
If, in addition, $\lambda >0$, the evolutionary system \eqref{NLS_diss} possesses a global attractor $\mathcal A\subset \Sigma$, i.e., $\mathcal A$ is
is invariant under the time-evolution associated to \eqref{NLS_diss} and such that
\[
\inf_{\phi \in \mathcal A} \| \psi(t) - \phi \|_{L^2(\R^d)} \stackrel{t\to +\infty}{\longrightarrow} 0. 
\]
More precisely, 
\[
\mathcal A=\big \{\psi_0 : \psi_0 = \psi(0) \mbox{ for some } \psi \in C({(-\infty}, \infty);\Sigma) \, \text{solution to \eqref{NLS_diss}} \big \} 
\]
is a connected compact set in $L^2(\R^d)$ and uniformly attracts bounded sets in $L^2(\R^d)$. 
Furthermore, for $\sigma \ge \frac{2}{d}$, $\mathcal A$ has finite Hausdorff and fractal dimensions which depend on the given parameters as described in Proposition \ref{prop:dimension}.
Finally, if $\mu < \frac{\omega d}{2}$ it holds $\mathcal A = \{ 0 \}$.
\end{theorem}

Here, $\Sigma$ is the physical energy space ensuring that $E_{\rm GP}(\psi(t))$ is finite. 
The assumption on $\sigma>0$ is thereby slightly more restrictive than the one for the usual $H^1$-subcritical nonlinearities (see Remark \ref{remH1} below). 
Note however, that we may always take $\sigma = 1$ in the above theorem which corresponds to the usual cubic nonlinearity. In addition, the condition $\omega>|\Omega|$ 
ensures that the confinement is stronger than the rotation, and thus, the system remains trapped for all times $t\ge 0$. 

As we shall see, neither the mass nor the (total) energy are conserved quantities of the time-evolution, but for $\lambda >0$, there are {\it absorbing balls} for $\psi$ in both 
the mass and the energy space, see Section \ref{sec:uniform} for a precise definition. The existence of a global attractor $\mathcal A$ therefore requires the presence of the nonlinearity 
and, of course, the presence of the confining potential $V$. Clearly, all stationary solutions $\varphi \in \Sigma$ of \eqref{stat_NLS} 
are members of $\mathcal A$. However, since for $\mu$ sufficiently large 
there are always at least two such solutions (namely, zero and the nontrivial energy minimizer) and since $\mathcal A$ is connected, it is unclear what the precise 
long-time behavior of \eqref{NLS_diss} is. 
Indeed, in the case of the GL equation for superconducting materials it is known \cite{TW} 
that the global attractor contains not only all possible steady state solutions, but also the heteroclinic orbits joining these steady states, and we 
consequently expect a similar behavior to also hold also in our model. 

Except in the case $\mu<\frac{\omega d}{2}$, the precise dependence of the dimension of $\mathcal A$ on the given physical parameters is not known. 
In Section \ref{sec:dimension} we shall prove that the Hausdorff dimension ${\rm dim}_{\rm H}(\mathcal A)\le m$, where $m$ depends in a 
rather complicated way on all the involved parameters. It is interesting, however, to check that $m\to +\infty$, as $|\Omega|\to \omega$. 
In other words, the influence of the rotation term potentially increases the dimension of the attractor. This is consistent with 
numerical and physical experiments on the creation of vortex lattices in rotating BEC. For a recent (non-rigorous) study which employs numerical simulations and asymptotic analysis 
to investigate the corresponding pattern formation mechanism, we refer to \cite{CGKT}. In fact, one easily observes that in the linear case ($\lambda =0$) the dynamics 
admits exponentially growing modes, cf. Section \ref{sec:linear} below for more details. It is argued in \cite{CGKT} that this type of instability mechanism is responsible for the 
nucleation of a large number of vortices at the periphery of the atomic cloud, as can be seen in physical experiments.\\

The proof of Theorem \ref{thmmain} will be done in several steps: First, we shall establish local (in-time) well-posedness of \eqref{NLS_diss} in Section \ref{sec:local} below.
Then, we will show how to extend this result to global in-time solutions in Section \ref{sec:global}, where we also prove that for $\mu < \frac{\omega d}{2}$ solutions decay to zero as $t\to +\infty$. 
The main technical step for the existence of an attractor is then to prove certain uniform bounds on the total mass and energy as done in Section \ref{sec:uniform}.
This  will allow us to conclude the existence of an absorbing ball and of a global attractor in Section \ref{sec:attractor}, where we shall also prove the announced estimates on the dimension. 
Finally, we collect some basic computations regarding the kernel of the linear semigroup in the appendix.

\section{Mathematical preliminaries}

In this section we shall collect several preliminary results to be used later on.

\subsection{Spectral properties of the linear Hamiltonian}\label{sec:linear}
In the following, we denote by
\begin{equation}\label{H}
H_\Omega := -\frac{1}{2} \Delta + V(x)  - \Omega L ,\quad x\in \R^d,  
\end{equation}
the linear Hamiltonian operator, with $V(x)$ given in \eqref{Vquadr}. Note that in the case without rotation, i.e. $\Omega=0$, the operator 
\begin{equation}\label{H_0}
H_0=\frac12 \left(- \Delta+\omega ^2 {|x|^2}\right),
\end{equation} is nothing but the 
(isotropic) quantum mechanical harmonic oscillator in, respectively, $d=2$, or $3$ spatial dimensions. The spectral properties of $H_0$ are well known \cite{Fl, T}: 
\begin{lemma}\label{lem:H}
$H_0$ is essentially self-adjoint on $C_0^\infty(\R^d)\subset L^2(\R^d)$ with compact resolvent. The 
spectrum of $H_0$ is given by $\sigma(H_0) = \{ E_{0,n}\}_{n\in \N}$, where
\[
E_{0,n} = \omega \Big( \frac{d}{2} +n-1 \Big),\quad n =1,2, \dots.
\]
In addition, the eigenvalue $E_{0,n}$ is $\left(\begin{matrix} d+n -2 \\ n-1 \end{matrix}\right)-$fold degenerate.
\end{lemma}
In particular, $E_{0,n} \ge E_{0,1}\equiv \frac{\omega d}{2}>0$, for all $n \in\N$. 
The associated eigenfunctions form a complete orthonormal basis of $L^2(\R^d)$. In $d=2$, they are explicitly given by \cite{Fl}:
\[
\chi^0_{n_1, n_2}(x_1, x_2) = f_{n_1}(x_1) f_{n_2}(x_2),\quad n_j\in \N,
\]
where  $n_1 + n_2 = n$ and the $f_{n_j}\in \mathcal S(\R)$ are the eigenfunctions of the one-dimensional harmonic oscillator, i.e., an appropriately normalized Gaussians times a Hermite polynomial of order $n_j-1$. 
An analogous formula holds in $d=3$ dimensions.

In the case with $\Omega\not =0$, we first note that the commutator $[H_\Omega, L]=0$, due to the rotational symmetry of the potential $V$. 
This implies that $H_\Omega$ and $L$ have a common orthonormal basis of eigenfunctions $\{\chi_n\}_{n\in \N_0}$, which can be obtained by taking appropriate linear combinations of the 
eigenvalues of $H_0$, see \cite{Fl}. An important assumption throughout this work, will 
be that $\omega>|\Omega|$, ensuring confinement of the BEC. 
In mathematical terms, this condition implies that the rotational term can be seen as a perturbation of the positive definite operator
$H_0$, such that $H_\Omega$ is still positive definite. In other words, we have that 
\begin{equation}\label{eigenvalue}
H_\Omega \chi_n = E_{\Omega,n} \chi_n,
\end{equation}
where the new eigenvalues $E_{\Omega, n}\in \R$ (indexed in increasing order) are related to the unperturbed $E_{0,n}$ via 
\[
\{E_{\Omega, n}, \ n \in \N\}= \{ E_{0,\ell} + m \Omega, \ -\ell+1 \le m\le \ell-1, \, \text{for}\, \ell \in\N\}.
\] 
In particular, under the assumption that $\omega > \Omega$, we still have: $E_{\Omega, n}\ge \frac{\omega d}{2}$, for all $n\in \N$. Thus, the ground state energy 
eigenvalue stays the same with and without rotation. \\

With these spectral data at hand, we can now look at the {\it linear} time-evolution ($\lambda =0$) associated to \eqref{dissGP}, i.e.
\begin{equation}\label{linGP}
(i\beta-\gamma) \partial_t \psi =H_\Omega \psi - \mu \psi.
\end{equation}
Using the fact that $\{\chi_n\}_{n\in \N}$ comprises an orthonormal basis, we can decompose the solution to this equation via
\begin{equation}\label{decomp}
\psi(t,x) =   \sum_{n \in \N} c_n(t) \chi_n (x),
\end{equation}
where $ \{ c_n (t) \}_{n\in \N} \in \ell^2$, i.e. $\sum |c_n(t)|^2 < +\infty$. In view of \eqref{eigenvalue}, \eqref{linGP} we find
\[
c_n(t) = c_n(0) \exp (-(i\beta + \gamma)(E_{\Omega, n}-\mu) t),
\]
In particular, the normalization $\beta^2+\gamma^2=1$ yields 
\[
\| \psi_n (t) \|^2_{L^2} \equiv \sum_{n=1}^\infty |c_n(t)|^2 = \sum_{n=1}^\infty |c_n(0)|^2 e^{- 2\cos\vartheta (E_{\Omega, n} - \mu) t},
\]
where we identify $\gamma =\cos \vartheta$.
For $ \vartheta \in (-\frac{\pi}{2}, \frac{\pi}{2})$ 
the right hand side exponentially decays to zero as $t\to+\infty$, provided $\mu < E_{\Omega, n}$, for all $n\in \N$. This is equivalent to saying that $\mu <E_{\Omega, 1}$.
On the other hand, if $c_n(0)=0$, then $c_n(t)=0$ for all $t>0$. Hence, given a $\mu > E_{\Omega, 1}$ the solution is exponentially decaying as long as 
the initial data is such that $c_n(0)=0$ for all $n\in \N$ for which $E_{\Omega, n} < \mu$. Otherwise, we have, in general, exponential growth of the 
$L^2$-norm of $\psi(t)$. 

\begin{remark}
In the case where we choose $\mu = E_{\Omega, m}$ for some fixed $m\in \N_0$, we see that the $|c_m(t)|^2 = |c_m(0)|^2$ is a conserved quantity of the linear time evolution. 
All higher modes exponentially decay towards zero, whereas all lower modes will exponentially increase. We consequently expect linear instability of stationary states of the nonlinear system.
\end{remark}

\subsection{Dispersive properties of the linear semi-group} 
In order to set up a well-posedness result for the nonlinear equation \eqref{dissGP}, we need to study the regularizing properties of the 
linear semigroup associated to $H_\Omega$, i.e.
\[
S_\Omega(t):= \exp\left(-e^{-i\vartheta} t  H_\Omega\right), \quad t\in \R_+,
\]
As usual we identify $S_\Omega(t)$ with its associated integral kernel via
\[
S_\Omega(t) f(x) = \int_{\R^d} S_\Omega(t,x,y) f(y) \, dy,\quad f \in L^2(\R^d).
\]
The following lemma states some basic properties of $S_\Omega(t)$ to be used later on.

\begin{lemma}\label{lem:kernel}
Let $\vartheta \in (-\frac{\pi}{2}, \frac{\pi}{2})$ and $t>0$. Then  
\begin{align}\label{SG_diss}
 S_\Omega(t,x,y) = \bigg(\frac{\omega}{2 \pi \sinh(e^{-i\vartheta} \omega t)}\bigg)^\frac{d}{2} \exp\left({\Phi(t,x,y)}\right),
\end{align}
where the pre-factor in front of the exponent is understood in terms of the principal value of the complex logarithm, and the phase function $F$ is given by
\begin{align*}
 \Phi(t,x,y) &=  -\frac{\omega}{\sinh(e^{-i\vartheta} \omega t)}\bigg(\frac{1}{2}(x^2+y^2)\cosh(e^{-i\vartheta}
\omega t) - \cosh(e^{-i\vartheta}\Omega t) (x_1 y_1 + x_2 y_2) \\&\qquad\qquad\qquad\qquad\qquad + i \sinh(e^{-i\vartheta} \Omega t) (x_2y_1 - x_1y_2) \bigg).
\end{align*}
Moreover, for $\omega>|\Omega|$, there exists $\delta>0 $ such that 
\begin{equation}\label{kernel_Lp}
 \|S_\Omega(t) f \|_{L^r} \le C\ t^{\frac{d}{2}(\frac{1}{r}-\frac{1}{q})} \| f \|_{L^q}
\end{equation}
and
\begin{equation}\label{kernel_Sigma}
  \|\nabla S_\Omega(t) f \|_{L^r} + \|x S_\Omega(t) f \|_{L^r},
\le C\ t^{-\frac{1}{2}+\frac{d}{2}(\frac{1}{r}-\frac{1}{q})} \| f \|_{L^q},
\end{equation}
for all $1\le q\le r \le \infty$ and all $0<t<\delta$, where the constants $C$ and $\delta$ only depend on $\vartheta$, $\omega$, and $\Omega$.
\end{lemma}

The proof of Lemma~\ref{lem:kernel} is a lengthy but straightforward calculation given in the Appendix. It is based on the well-known Mehler formula, cf. \cite{Car09}, and 
a time-dependent change of coordinates introduced in \cite{AMS}. 

\begin{remark} 
The decay estimates stated above are the same as for the heat equation. Indeed, $S_\Omega(T)$ may be viewed as an analytic
perturbation of the classical heat semigroup. In the case without potential and without rotation, i.e. $\Omega = \omega = 0$, similar estimates 
have been derived in, e.g., \cite{CDW1}. \end{remark}

\section{Local well-posedness} \label{sec:local}

In this section we set up a local well-posedness result for the initial value problem \eqref{NLS_diss}. In order to do so, we use Duhamel's formula to rewrite \eqref{NLS_diss} as
\begin{equation}\label{mild_GP}
 \psi(t) = S_\Omega(t)\psi_0 - e^{-i\vartheta} \int_0^t S_\Omega(t-\tau) \big( \lambda |\psi(\tau)|^{2\sigma}- \mu \big) \psi(\tau)  \;d\tau ,
\end{equation}
for all $t\in[0,T]$. Here, and in the following, we denote $\psi(t) \equiv \psi(t,\cdot)$. We shall work in the physical energy space given by
\begin{equation*}
\Sigma=\big\{f\in H^1(\R^d)\; : \; |x|f\in L^2(\R^d)\big\}
\end{equation*}
and equipped with the norm
\[
\| f \|_{\Sigma}^2:= \| f \|_{L^2}^2 + \| \nabla f \|_{L^2}^2 + \| x f \|_{L^2}^2.
\]
The estimates on the semi-group $S_\Omega(t)$ stated in Lemma \ref{lem:kernel} allow us to infer the following 
result (which is similar to those in \cite{DGL, GV1}).
\begin{proposition}\label{prop:loc_ex}
Let $\lambda, \mu \in \R$, $\vartheta \in (-\frac \pi  2 ,\frac \pi  2 )$, $\omega>|\Omega|$, and $d\in \{2,3\}$. 
\begin{itemize}
\item[(i)] Let $p > \max (\sigma d, 2\sigma +1)$ and $\psi_0\in L^p(\R^d)$. Then there exists a time $T>0$ and a unique solution $\psi \in C([0,T];L^p(\R^d))$ to \eqref{NLS_diss}, depending continuously on the initial data.
\item[(ii)] If, in addition, $0 < \sigma < \frac{d}{2(d-2)}$ and $\psi_0 \in \Sigma $, then there exists a $T^*>0$ such that the solution from ${\rm (i)}$ satisfies 
\[
 \psi \in C([0,T^*]; \Sigma).
\]
Moreover, the solution is maximal in the sense that either $T^*=+\infty$, or the following blow-up alternative holds:
\begin{equation*}\label{blowup_alt}
\lim_{t\to T_-^*} \| \psi(t)\|_{\Sigma} = \infty.
\end{equation*}

\end{itemize}
\end{proposition}

\begin{proof}
The proof is based on a fixed point argument using Duhamel's formula and the properties of the semigroup $S_\Omega(t)$. To this end, we first note that the term $\mu\psi$ is of no importance here, as it can always be added in a subsequent step (in fact,
we could have included it in the kernel of $S_\Omega(t)$). Hence let us assume that $\mu=0$ for notational convenience. 

To prove (i), we will show that the mapping
\[
\psi \mapsto \Xi(\psi)(t) := S_\Omega(t)\psi_0 - e^{-i\vartheta} \int_0^t S_\Omega(t-\tau) \big(\lambda |\psi(\tau)|^{2\sigma}\psi(\tau)\big) \;d\tau
\]
is a contraction in the space
\[
 X_T:=\big\{ \psi\in C([0,T];L^p(\R^d)) \ : \ \|\psi\|_{L^\infty(0,T;L^p)} \le 2 \|\psi_0\|_{L^p} \big\}
\]
for small enough $T>0$. To do so, we can use the kernel estimate \eqref{kernel_Lp} with the following choice of parameters:
\[
\begin{split}
r=p\ge 2\sigma+1, \quad q=\frac{p}{2\sigma+1},\qquad  \text{when} \qquad d&=2,\\
r=p > \max (\sigma d, 2\sigma +1), \quad q=\frac{p}{2\sigma+1},\qquad  \text{when} \qquad d&=3.
\end{split}  
\]
Note that any such a choice of $p$ implies that $d\sigma <p$. One can also see that $1\le q  \le r\le \infty$ in both cases.
This yields
\begin{align*}
 \|\Xi(\psi)(t)\|_{L^p} &\le \|\psi_0\|_{L^p} + \lambda \int_0^t \big\| S_\Omega(t-\tau)\big(|\psi(\tau)|^{2\sigma}\psi(\tau)\big)\big\|_{L^p} \;d\tau\\
 &\le \|\psi_0\|_{L^p} + C \int_0^t (t-\tau)^{-{d\sigma}/{p}} \|\psi(\tau)\|_{L^p}^{2\sigma+1} \;d\tau\\
 &\le \|\psi_0\|_{L^p} + C \|\psi\|_{L^\infty(0,T;L^p)}^{2\sigma+1} 
 \int_0^T \tau^{-{d\sigma}/{p}} \;d\tau \end{align*}
Since $\sigma < \frac p d$, the remaining integral is finite and hence, 
\[
 \|\Xi(\psi)(t)\|_{L^p} \le  \|\psi_0\|_{L^p} + C T^{1-\frac{d\sigma}{p}} \|\psi\|_{L^\infty(0,T;L^p)}^{2\sigma+1} ,
\]
where $1-\frac{d\sigma}{p}>0$. Thus, for $T>0$ sufficiently small, we conclude that $\Xi$ indeed maps $X_T$ onto itself.
Likewise it holds that for two solutions $\psi$ and $\tilde \psi$
\begin{align*}
& \|\Xi(\tilde\psi)(t)-\Xi(\psi)(t)\|_{L^p}\\ &\le \lambda \int_0^t
\big\|S_\Omega(t-\tau)\big(|\tilde\psi(\tau)|^{2\sigma}\tilde\psi(\tau)-|\psi(\tau)|^{2\sigma}\psi(\tau)\big)\big\|_{L^p} \;d\tau\\
 &\le C \int_0^t (t-\tau)^{-{d\sigma}/{p}} \big(\|\tilde\psi(\tau)\|_{L^p}^{2\sigma} + \|\psi(\tau)\|_{L^p}^{2\sigma}\big) \|
\tilde\psi(\tau)-\psi(\tau) \|_{L^p} \;d\tau\\
 &\le C T^{1-\frac{d\sigma}{p}} \|\psi\|_{L^\infty(0,T;L^p)}^{2\sigma} \| \tilde\psi-\psi
\|_{L^\infty(0,T;L^p)},
\end{align*}
which shows that $\Xi$ is a contraction for $T>0$ sufficiently small. 

\medskip

To prove (ii), we first note that by Sobolev imbedding $\Sigma \hookrightarrow L^p(\R^d)$, for $p< p^*=\frac{2d}{d-2}$
when $d=3$ and $p<\infty$ when $d=2$, respectively. Thus $\Sigma\cap L^p(\R^d) = \Sigma$ for $p<p^*$. We now want to show that for $0 < \sigma < \frac{d}{2(d-2)}$, the 
$\Sigma$ norm of the solution is controlled by an appropriately chosen $L^p$ norm satisfying $p<p^*$ and the conditions in part ${\rm (i)}$. 

The first step to do so, relies on 
appropriate
expressions for the commutators $[\nabla, S_\Omega(t)]$ and $[x, S_\Omega(t)]$. At least formally, it holds that
\[
 -e^{i\vartheta}\partial_t [\nabla, S_\Omega(t)] = [\nabla, H_\Omega S_\Omega(t)] = H_\Omega [\nabla, S_\Omega(t)] + [\nabla, H_\Omega] S_\Omega(t),
\]
and one easily computes $$[\nabla, H_\Omega] = \nabla V + i \Omega \nabla^\bot = \omega^2 x + i \Omega \nabla^\bot, $$ in view of \eqref{Vquadr}. Hence, Duhamel's formula and the fact that $[\nabla, S_\Omega(0)]=0$ 
imply
\[
 [\nabla, S_\Omega(t)]  =-e^{-i\vartheta} \int_0^t S_\Omega(t-\tau) \big( \omega^2 x + i \Omega \nabla^\bot \big) S_\Omega(\tau)\;d\tau.
\]
Using the fact that
$$
[x, H_\Omega] = \nabla - i \Omega x^\bot ,
$$
we likewise obtain
\[
 [x, S_\Omega(t)]  = -e^{-i\vartheta} \int_0^t S_\Omega(t-\tau) \big( \nabla - i \Omega x^\bot \big) S_\Omega(\tau)\;d\tau.
\]
Straightforward calculations then yield
\begin{equation*}\label{nabla_psi}
 \nabla \psi(t) = S_\Omega(t) \nabla \psi_0 -e^{-i\vartheta} \int_0^t S_\Omega(t-\tau) \left(\lambda \nabla (|\psi|^{2\sigma} \psi) +(\omega^2 x + i \Omega \nabla^\bot)\psi \right)(\tau) \;d\tau 
 \end{equation*}
as well as 
\begin{equation*}
\label{x_psi}
 x \psi(t) = S_\Omega(t) x \psi_0 -e^{-i\vartheta} \int_0^t S_\Omega(t-\tau) \left(\lambda x |\psi|^{2\sigma} \psi +  ( \nabla - i \Omega x^\bot) \psi\right)(\tau)\, d\tau.
\end{equation*}
We consequently expect that the combination of $\psi, x\psi, \nabla \psi$ will form a closed set of estimates (a fact already observed in \cite{AMS}).

It follows that the $\Sigma$ norm of $\psi$ is controlled by its $L^p$ norm.
For instance, choose $r=2$ and $q$ such that
$$\frac 1q=\frac{2\sigma}{p}+\frac12, \quad \text{ in \eqref{kernel_Lp}.}$$ Note that the condition $q\ge 1$ is equivalent to $p\ge 4\sigma$ and thus we require the existence of a $p$ such that
\[
\max(4\sigma, \sigma d, 2\sigma +1) < p < \frac{2d}{d-2}.
\]
For $d\in \{ 2,3\}$ one can readily check that $0< \sigma <  \frac{d}{2(d-2)}$ ensures the existence of such a $p$. It follows that
\begin{align*}
&\, \Big\|\int_0^t S_\Omega(t-\tau) \nabla \big(|\psi|^{2\sigma} \psi\big)(\tau) \;d\tau \Big\|_{L^2} \\
&\, \le C \|\psi\|_{L^\infty(0,T;L^p)}^{2\sigma} \|\nabla \psi\|_{L^\infty(0,T;L^2)} 
 \int_0^t (t-\tau)^{-\frac{d\sigma}{p}}\;d\tau.
\end{align*}
On the other hand, the linear terms can be estimated with $r=q=2$ in \eqref{kernel_Lp}, to obtain
\begin{align*}
\| \nabla \psi \|_{L^\infty(0,T;L^2)} \le & \,  \| \nabla \psi_0 \|_{L^2}+ C_1 T^{1-\frac{d\sigma}{p}} \|\psi\|_{L^\infty(0,T;L^p)}^{2\sigma}  \| \nabla \psi\|_{L^\infty(0,T;L^2)}  \\
& \, +C_2 T \left(  \| x \psi\|_{L^\infty(0,T;L^2)}  + 
  \| \nabla \psi\|_{L^\infty(0,T;L^2)}  \right)\\
  \le & \,  \| \nabla \psi_0 \|_{L^2}+ C_3 T^{1-\frac{d\sigma}{p}} \|\psi\|_{L^\infty(0,T;L^p)}^{2\sigma}  \| \nabla \psi\|_{L^\infty(0,T;L^2)}.
\end{align*}
for $T>0$ sufficiently small (depending on the size of $\|\psi\|_{L^\infty L^p}$). Similar arguments for $\psi$ and $x\psi$ imply
\begin{align*}
\| \psi \|_{L^\infty(0,T;\Sigma)} \le \| \psi_0 \|_{\Sigma}+ C T^{1-\frac{d\sigma}{p}} \|\psi\|_{L^\infty(0,T;L^p)}^{2\sigma}  \| \psi\|_{L^\infty(0,T;\Sigma)}  .
\end{align*}
Choosing $T>0$ even smaller, if necessary, the second term on the right hand side can be absorbed on the left hand side and we are done. 
As before, this inequality also applies to the differences of two solutions $\psi, \tilde \psi$, which yields the continuity of $\psi$ in $\Sigma$. 

We denote by $T^*>0$ the maximal time of existence in $\Sigma$. This is always less than or equal to $T>0$, the maximal time of existence in $L^p(\R^d)$. 
To prove the blow-up alternative, assume by contradiction that $T^*<\infty$, and $\| \psi(t, \cdot)\|_{\Sigma}$ remains bounded for $t\in [0, T^*]$. Then, by Sobolev imbedding 
$\| \psi(t, \cdot)\|_{L^p}$ also remains bounded and thus, we can restart the local existence argument in $\Sigma$ leading to a contradiction. 
\end{proof}

\begin{remark}\label{remH1}
Unfortunately, our method of proof does not yield existence of solutions for the full $H^1$-subcritical regime, i.e., $\sigma < \frac{2}{d-2}$. We expect that 
this is only a technical issue that can be overcome using a different approach (for example, by using ideas from \cite{GV1}, or by 
generalizing the space-time estimates of \cite{Bax} to $S_\Omega$). Note, however, that our slightly more restrictive condition $\sigma < \frac{d}{2(d-2)}$ still allows to take $\sigma =1$ in 
$d=3$. Hence, the physically most relevant case of a cubic nonlinearity is covered.
\end{remark}

\section{Global existence and asymptotic vanishing of solutions} \label{sec:global}

In this section, we shall first prove the global existence of solutions in the energy space before showing that for any choice of $\mu < E_{\Omega,1}$, the 
solutions asymptotically vanish as $t\to +\infty$.

\subsection{Global existence}

In order to prove global well-posedness of \eqref{NLS_diss}, we 
will need to collect some useful a-priori estimates. To this end, we denote for $\psi\in \Sigma$ the {\it total mass} by
\begin{equation}\label{mass}
M(\psi):= \| \psi \|_{L^2}^2,
\end{equation}
and the {\it total energy} by
\begin{equation}\label{energy}
E(\psi):= \int_{\R^d} \bigg( \frac{1}{2} |\nabla \psi|^2 + V(x) |\psi|^2 + \frac{\lambda}{\sigma +1}|\psi |^{2\sigma+2} -  \Omega \overline{\psi}
L \psi \bigg) \;dx.
\end{equation}
The latter is nothing but the sum of the kinetic, potential, nonlinear potential, and rotational energy. Clearly, for $\psi \in \Sigma$, 
Sobolev's imbedding implies that all the terms in $E(\psi)$ are finite, provided $\sigma < \frac{2}{d-2}$ (and hence also for our range of $\sigma$). 
For simplicity of notation, we will write $E(t) \equiv E(\psi(t,\cdot))$ and likewise for $M(t)$, whenever we compute the 
mass and energy of the time-dependent solution $\psi(t,x)$ to \eqref{NLS_diss}.
In addition, the {\it free energy} is 
given by
\begin{equation}\label{free_energy}
F(\psi):=E(\psi) - \mu M(\psi).
\end{equation}

In the case of the usual Gross-Pitaevskii equation, i.e. $\vartheta = \pm \frac \pi 2$, one finds, that both $M(t)=M(0)$ and $E(t)=E(0)$ are conserved in time \cite{AMS}. 
In our dissipative model this is no longer the case. Instead we have the following result, which can be seen as an extension of 
some well-known identities proved for the classical GL equation, cf. \cite{DGL, GV2, Te, Wu}.

\begin{lemma}\label{lem:Lyapunov}
Let $\sigma <\frac{d}{2(d-2)}$ and $\psi \in C([0,T]; \Sigma)$ be a solution to \eqref{NLS_diss}. Then the following identities hold:
\begin{equation}\label{eq:mass}
M(t) + 2\cos \vartheta \int_0^t  \left( E(s) + \frac{\lambda\sigma}{\sigma+1} \| \psi(s,\cdot) \|_{L^{2\sigma+2}}^{2\sigma+2} - \mu M(s)\right)  ds =M(0),
\end{equation}
and
\begin{equation}\label{eq:energy}
F(t) + 2\cos \vartheta \int_0^t  \int_{\R^d} | \partial_t \psi(s,x)|^2\, dx \, ds= F(0).
\end{equation}
In particular, for $\vartheta \in (-\frac \pi  2 ,\frac \pi  2 )$, the free energy $F(\psi)$ is a non-increasing functional along solutions of \eqref{NLS_diss}.

\end{lemma}
\begin{proof} In a first step, let us assume sufficient regularity (and spatial decay) of $\psi$, such that all the following calculations are justified. Then, 
as in the case of the usual NLS, 
identity \eqref{eq:mass} is obtained by multiplying \eqref{NLS_diss} by $\bar \psi$, integrating with respect to $x\in \R^d$ and taking the real part of the resulting expression (see, e.g., \cite{AMS, Caze}). 
This yields
\begin{equation}\label{mass-derivative}
\frac{d}{dt} M(t) = -2\cos \vartheta \left( E(t) + \frac{\lambda\sigma}{\sigma+1} \| \psi(t) \|_{L^{2\sigma+2}}^{2\sigma+2} - \mu M(t)
\right)
\end{equation}
which directly implies \eqref{eq:mass} after an integration in time. Similarly, after multiplying \eqref{NLS_diss} by $\partial_t \bar \psi$, 
integrating with respect to $x$, and taking the real part, we obtain
\begin{equation}\label{energy-derivative}
 \frac{d}{dt} \big(E(t) - \mu M(t)\big) = -2\cos \vartheta \int_{\R^d} | \partial_t \psi(t,x)|^2\, dx,
\end{equation}
which yields \eqref{eq:energy} after integration w.r.t. time. 

The second step then consists of a classical density argument (cf. \cite{CDW1}), which, together with the fact that 
$\psi(t)$ depends continuously on the initial data $\psi_0\in \Sigma$, 
allows us to extend \eqref{eq:mass} and \eqref{eq:energy} to the case of general solutions $\psi \in C([0,T;\Sigma)$. 
Finally, we note that for $\vartheta \in (-\frac \pi  2 ,\frac \pi  2 )$ we have $\cos \vartheta >0$, and thus 
\eqref{eq:energy} directly implies that $F(t) \le F(0)$, for all $t\ge 0$.
\end{proof}

Having in mind that $\psi \in C([0,T], \Sigma)$ the assumption on $\sigma$ implies (via Sobolev imbedding) that the integrand 
appearing in identity \eqref{eq:mass} is a continuous function of time. 
The fundamental theorem of calculus therefore allows us to differentiate \eqref{eq:mass} w.r.t. $t$ and consequently use the 
differential inequality \eqref{mass-derivative}. However, the same is not true for \eqref{eq:energy}, i.e., we cannot use \eqref{energy-derivative}, 
since at this point we do not know wether $\partial_t \psi  \in C([0,T; L^2(\R^d))$ holds true. This fact will play a role in some of the proofs given below. \\

Another preliminary result, to be used several times in the following, is the fact that under our assumptions on the parameters $\omega, \Omega, \lambda, \sigma$, 
the energy is indeed non-negative.

\begin{lemma} \label{lem:Ebound}
Let $\omega>|\Omega|$, $\lambda \ge 0$, and $\sigma < \frac{2}{d-2}$. Then for any 
$u \in \Sigma$ there exists a constant $c=c(\omega, \Omega, \lambda , \sigma)>0$,  such that
such that 
\[
\|\nabla u \|_{L^2}^2 + \| x u \|_{L^2}^2  + \| u \|_{L^{2\sigma +2}}^{2\sigma+2} \le cE(u).
\]
\end{lemma}
\begin{proof}
Since $\lambda\ge0$, the only possibly negative term within $E(u)$ is given by the rotational energy. However, since 
$\Omega^2/\omega^2=:\epsilon < 1$, Young's inequality applied to~\eqref{eq:angular_momentum} yields the pointwise interpolation estimate
\[
\big| \Omega \overline{u} L u \big| \le \frac{\omega^2}{2}  |x^\bot|^2 | u |^2 \, dx +\frac{\epsilon}{2} |\nabla^\bot u|^2 \le V(x) |u|^2 + \frac{\epsilon}{2} |\nabla u |^2.
\]
We therefore can bound the energy from below via
\begin{equation*}\label{grad_by_en}
 0 \le \frac{1-\epsilon}{2} \|\nabla u \|_{L^2}^2 +  \frac{\lambda\sigma}{\sigma+1} \| u \|_{L^{2\sigma+2}}^{2\sigma+2} \le E(u).
\end{equation*}
Analogously, we have
\[
 0 \le \frac{1-\epsilon}{2} \|x u \|_{L^2}^2+  \frac{\lambda\sigma}{\sigma+1} \| u \|_{L^{2\sigma+2}}^{2\sigma+2} \le E(u).
\]
Combining these two estimates then yields the desired result with a constant
\[
c=\frac{4}{\min\{1-\epsilon, \frac{2\lambda \sigma}{\sigma +1}\}}.
\]
Note that  $c\to +\infty$ as $|\Omega|\to \omega$.
\end{proof}

The mass/energy-relations stated in Lemma \ref{lem:Lyapunov} can now be used to infer global existence of solutions in the 
case of {\it defocusing} case $\lambda > 0$.

\begin{proposition} \label{prop:global} Let $\omega>|\Omega|$, $\vartheta \in (-\frac \pi  2 , \frac \pi  2 )$, $\lambda \ge 0$, and $\sigma < \frac{d}{2(d-2)}$. 
Then, for any $\psi_0 \in
\Sigma$
there exists a unique global-in-time solution $ \psi \in C([0,\infty); \Sigma)$ to \eqref{NLS_diss}.
\end{proposition}
\begin{proof}
In view of the blow-up alternative stated in Proposition \ref{prop:loc_ex}, all we need to show is that the $\Sigma$-norm remains bounded for all $t\ge 0$. 
Lemma \ref{lem:Ebound} implies that this is the case, as soon as we we can show that both $M(t)$ and $E(t)$ are bounded. 
In order to do so, we first consider the case $\mu <0$ and recall that $\cos \vartheta >0$ for $\vartheta \in (-\frac \pi  2 , \frac \pi  2 )$. In this case 
identity \eqref{eq:energy} implies
\[
E(t) + |\mu| M(t) \le F(0) <+\infty,
\]
and since both $E(t)$ and $M(t)$ are non-negative, we directly infer the required bound on the mass and energy.

On the other hand, for $\mu \ge 0$, identity \eqref{eq:mass} yields (since $\lambda \ge 0$)
\[
M(t) \le M(0) + 2 \mu \cos \vartheta \int_0^t M(s) \, ds,
\]
and hence, Grownwall's lemma implies
\begin{equation}\label{Mbound}
M(t) \le M(0)\left( 1+ 2 \mu t \cos \vartheta \, e^{2\mu t \cos \vartheta } \right).
\end{equation}
Using this estimate in identity \eqref{eq:energy} we obtain
\[
 E(t)  \le F(0)  + \mu M(t) \le E(0)  + 2 \mu^2 t \cos \vartheta  M(0)  e^{2\mu t \cos \vartheta }  .
\]
The right hand side is finite, for all $t\ge0$ and thus, the assertion is proved.
\end{proof}

\begin{remark}
The global in-time strong solutions constructed above are of the same type as the corresponding solutions for NLS with quadratic potentials, cf. \cite{AMS, Car05}. 
It is certainly possible to, alternatively, construct global weak solutions to \eqref{NLS_diss} as has been done for the usual GL model in, e.g., \cite{DGL, GV1}. 
But since we consider the equation \eqref{NLS_diss} as a 
toy model describing possible relaxation phenomena in the mean-field dynamics of BEC, we have decided to remain as close as possible to the corresponding NLS theory. 
In particular, we do not make any use of the strong smoothing property of the linear (heat type) semigroup $S_\Omega(t)$ for $\vartheta \in (-\frac{\pi}{2}, \frac \pi 2)$. 
We finally note that our set-up makes it possible to directly generalize the inviscid limit results of \cite{Wu} to our model.
\end{remark}

\subsection{Asymptotically vanishing solutions}

The discussion in Section \ref{sec:linear} shows that solutions to the linear time evolution $\lambda =0$ asymptotically vanish, provided 
$\mu < E_0$, i.e., the lowest (positive) energy eigenvalue of $H_\Omega$. 
We shall prove that the same is true for in the nonlinear case $\lambda >0$.

\begin{proposition} Let $\vartheta \in (-\frac \pi  2 , \frac \pi  2 )$, $\lambda \ge 0$, $\omega>|\Omega|$, and $\psi \in C([0,\infty), \Sigma)$ be a solution of \eqref{NLS_diss} with 
$\mu<E_{\Omega, 1}=\frac{\omega d}{2}$. Then
\[
\lim_{t\to +\infty} \| \psi(t)\|_{L^2} = 0,
\]
exponentially fast.
\end{proposition}

\begin{proof} For solutions $\psi \in C([0,\infty), \Sigma)$ we are allowed to use the differential inequality \eqref{mass-derivative}, which together with the fact that $\lambda \ge0$ implies
\[
\frac{d}{dt} M(t) \le -2\cos \vartheta \left( E(t)  - \mu M(t)\right).
\]
Decomposing $\psi(t,x)$ in the form \eqref{decomp}, and dropping the nonlinear term $\propto \lambda$ within $E(t)$, 
then allows us to rewrite this inequality as
\[
\frac{d}{dt} M(t) \le -2\cos \vartheta \sum_{n=0}^\infty (E_{\Omega,n}-\mu) |c_n(t)|^2 \le -2\cos \vartheta (E_{\Omega,0} - \mu)M(t),
\]
since $E_{\Omega,n}-\mu \ge E_{\Omega, 0} - \mu >0$, and $M(t) =  \sum_{n=0}^\infty  |c_n(t)|^2$.
The inequality above can thus be rewritten as
\[
\frac{d}{dt} \left(e^{+2 t\cos \vartheta (E_{\Omega,0} - \mu) } M(t) \right)  \le 0,
\]
which after an integration in time implies 
\[
M(t)\le  M(0) e^{-2 t\cos \vartheta (E_{\Omega,0} - \mu) } \xrightarrow{t\to +\infty}0,
\]
since $\vartheta \in (-\frac \pi  2 , \frac \pi  2 )$. 
\end{proof}

At this point, it is unclear if the decay rate given above is indeed sharp.

\begin{remark} In the case where $\mu <0$, one does not need to use the decomposition of $\psi$ via the spectral subspaces of $H_\Omega$, at the expense of a slightly worse decay rate. 
Indeed, for $\mu<0$, the inequality \eqref{mass-derivative} directly yields
\[
\frac{d}{dt} M(t) \le - 2 |\mu| \cos \vartheta M(t),
\]
and thus
\[
 M(t) \le M(0) e^{-2 t|\mu| \cos \vartheta} , \quad \forall t \ge 0.
\]
Note that for $\mu <0$ there are no nontrivial steady states $\varphi(x)\not =0$, satisfying \eqref{stat_NLS}. This can be seen by 
multiplying equation \eqref{stat_NLS} with $\bar \varphi$, integrating in $x\in \R^d$, 
and recalling the restriction $\omega>\Omega\ge0$, which implies that $\mu$ has to be non-negative.
\end{remark}

\section{Bounds on the mass and energy}\label{sec:uniform}

In this section we shall prove the existence of absorbing balls in both $L^2(\R^d)$ and $\Sigma$ for solutions to \eqref{NLS_diss}.
In view of the discussion on the linear model, cf. Section \ref{sec:linear}, this might seem surprising, given that for general $\mu>0$ we can expect exponentially growing modes. 
However, we shall see that for $\lambda>0$, the nonlinearity, combined with the confining potential, mixes the dynamics in a way that 
makes it possible to infer a uniform bound on the mass and energy, and hence on the $\Sigma$--norm of the solution.
To this end, the following lemma is the key technical step.

\begin{lemma}\label{lem:glob_bound}
Let $\lambda > 0$, $\omega > |\Omega|$ and $0 < \sigma < \frac{d}{2(d-2)}$. Then there exists a constant $C=C(\omega, \Omega, \lambda, \sigma)>0$, such that
\[
 M(\psi) \le C E(\psi)^{\frac{\sigma\theta + 1}{\sigma+1}} ,\quad \text{with $\theta = \frac{d\sigma}{2\sigma + 2 + d\sigma}$.}
 \]
\end{lemma}
\begin{proof}
The proof of this result relies on the following localization property: 
 For all $d\ge1$ and all $p\ge2$ and any $f\in C^\infty_0(\R^d)$:
\begin{equation}\label{localize}
 \| f \|_{L^2(\R^d)} \le { 2} \| x f \|_{L^2(\R^d)}^{\theta} \| f \|_{L^p(\R^d)}^{1-\theta},
\end{equation}
with 
\[
\theta = \frac{d(\frac{1}{2}-\frac{1}{p})}{1+d(\frac{1}{2}-\frac{1}{p})} = \frac{d(p-2)}{2p + d(p-2)}.
\]
In order to show this, let $B_R$ denote the ball around the origin of radius $r>0$. We rewrite 
\begin{align*}
 \|f\|_{L^2(\R^d)} &= \|f\|_{L^2(B_r)} + \|f\|_{L^2(\R^d\setminus B_r)} \le r^{d(\frac{1}{2} - \frac{1}{p})} \|f\|_{L^p(B_r)} + \frac{1}{r} \|x f\|_{L^2(\R^d\setminus B_r)}\\
 &\le r^{d(\frac{1}{2} - \frac{1}{p})} \|f\|_{L^p(\R^d)} + \frac{1}{r} \|x f\|_{L^2(\R^d)}.
\end{align*}
The right-hand side is minimal if both summands are of the same order of magnitude, i.e.
\[
 r^{1 + d(\frac{1}{2} - \frac{1}{p})} = \frac{\|x f\|_{L^2(\R^d)}}{\|f\|_{L^p(\R^d)}}.
\]
With this choice of $r$, the estimate \eqref{localize} follows and a density argument allows to extend it to any $f\in \Sigma$. 
Specifying $p=2\sigma+2$, consequently yields
\begin{equation}\label{proof_glob_bound}
\|\psi \|_{L^2}^2 \le { 2} \bigg( \int_{\R^d} |x|^2 |\psi(x)|^2 \;dx \bigg)^{\theta} \bigg( \int_{\R^d} |\psi(x)|^{2\sigma+2} \;dx \bigg)^{\frac{1-\theta}{\sigma+1}},
\end{equation}
where $\theta = \frac{d\sigma}{2\sigma + 2 + d\sigma}.$ 
In view of Lemma \ref{lem:Ebound}, both factors on the right hand side of \eqref{proof_glob_bound} are bounded by the energy. 
More precisely, 
\[
 M(\psi) \le 2(c E(\psi))^{\theta + \frac{1-\theta}{\sigma+1}} = C E(\psi)^{\frac{\sigma\theta+1}{\sigma+1}},
\]
where $C=2c^{\frac{\sigma\theta+1}{\sigma+1}}$ and $c=c(\omega, \Omega, \lambda, \sigma)>0$ is the constant from Lemma~\ref{lem:Ebound}. 
\end{proof}
\begin{remark}
Note that in order to infer this bound one needs 
the presence of {\it both} the confinement and the nonlinearity, since the proof requires $\sigma>0$, $\lambda>0$ and $\omega>0$. Moreover, one checks that $C\to +\infty$, as $|\Omega| \to \omega$.
\end{remark}

With this result in hand, we can deduce global bounds on $M(t)$ and $E(t)$ along solutions of \eqref{NLS_diss}.
\begin{proposition}\label{prop:Ebound}
 Let $\psi \in C([0,\infty), \Sigma)$ be a solution to \eqref{NLS_diss} with $\vartheta \in (-\frac \pi  2 ,\frac \pi  2 )$.
 Under the assumptions of Lemma \ref{lem:glob_bound}, if additionally $\mu>0$, there exists a constant $K=K(\omega, \Omega, \sigma, \lambda, \mu)>0$, independent of time, 
 such that
\[
 E(t) \le K + e^{- t \mu  \cos \vartheta} E(0) , \quad \forall \, t\ge 0.
\]
\end{proposition}
\begin{proof}
We first note that Lemma \ref{lem:glob_bound} and the differential inequality \eqref{mass-derivative} imply
\[
\frac{d}{dt}M(t) \le -2\cos \vartheta\, E(t) + C \mu E(t)^{\frac{\sigma\theta + 1}{\sigma+1}}   .
\]
Now, for any $\vartheta \in (-\frac \pi  2 ,\frac \pi  2 )$ and $\tilde\theta=  \frac{\sigma\theta + 1}{\sigma+1}$, by Young's inequality, we obtain
\[
 E(t)^{\tilde\theta} \le \frac{\cos\vartheta}{C\mu} E(t) + (1-\tilde \theta) \left(\frac{ C \mu \tilde \theta}{\cos \vartheta}\right)^{\frac{\tilde \theta} {1-\tilde \theta}} = 
 \frac{\cos \vartheta}{C\mu}  \, E(t)+\tilde C,
\]
where $\tilde C>0$, depends on all the parameters involved, but not on time. Thus, we have
\[
\frac{d}{dt}M(t) \le - \cos\vartheta E(t) + \mu C \tilde C.
\]
On the other hand, identity \eqref{eq:energy} implies 
\[
E(t) -E(t_0) \leq \mu M(t)-\mu M(t_0), \qquad 0 \leq t_0 \leq t,
\]
and hence
\[
E(t)-E (s) \le \int_{s}^t (-  \mu \cos \vartheta \, E(\tau)  + \mu^2 C \tilde C )\, d\tau, \qquad 0 \leq s \leq t,
\]
as well as
\[
E(t)-E (s) \ge \int_{s}^t (-  \mu \cos \vartheta \, E(\tau)  + \mu^2 C \tilde C )\, d\tau, \qquad 0 \leq t \leq s.
\]

Now, given any positive bump function $\chi \in C_0^{\infty}((t-\epsilon,t+\epsilon))$, such that
$\chi' \ge 0$ on $(t-\epsilon, t)$ and $\chi' \le 0$ on $(t,t+\epsilon)$, we multiply by $\chi'(s)$ and integrate in $s$, to obtain
\[
\begin{split}
\int_{t-\epsilon}^{t+\epsilon} [E(t)-E(s)]\chi'(s)\, ds \le& \int_{t-\epsilon}^{t+\epsilon} \int_s^t 
(-  \mu \cos \vartheta \, E(\tau)  + \mu^2 C \tilde C ) \chi'(s) \, d\tau \, ds\\
 =& \int_{t-\epsilon}^{t} \int_{t-\epsilon}^\tau 
(-  \mu \cos \vartheta \, E(\tau)  + \mu^2 C \tilde C ) \chi'(s) \, ds \, d\tau \\
&-\int_{t}^{t+\epsilon} \int_\tau^{t+\epsilon} 
(-  \mu \cos \vartheta \, E(\tau)  + \mu^2 C \tilde C ) \chi'(s) \, ds \, d\tau \\
 =& \int_{t-\epsilon}^{t+\epsilon} 
(-  \mu \cos \vartheta \, E(\tau)  + \mu^2 C \tilde C ) \chi(\tau) \, d\tau.\\
\end{split}
\]
A similar computation gives the same inequality for a negative bump function function $\chi \in C_0^{\infty}((t-\epsilon,t+\epsilon))$, such that
$\chi' \le 0$ on $(t-\epsilon, t)$ and $\chi' \ge 0$ on $(t,t+\epsilon)$. Since an  arbitrary test function can be written as a linear
combination of positive and negative bump functions, we have 
\[
-\int_{t_0}^t E(\tau) \chi'(\tau) \, d \tau \le \int_{t_0}^t \left(-  \mu \cos \vartheta \, E(\tau)  + \mu^2 C \tilde C \right) \, \chi(\tau) \, d\tau,
\]
for any $\chi \in C_0^{\infty}((t_0,t))$. Here, we have also used the fact that $\chi$ has compact support on $(t_0,t)$. 
Choosing $\chi(\tau)= e^{\mu \tau \cos \vartheta } \phi(\tau)$ with $\phi \in C_0^{\infty}((t_0,t))$, we obtain
\[
-\int_{t_0}^t E(\tau) \left(e^{\mu \tau \cos \vartheta } \phi(\tau)\right)' \, d \tau \le \int_{t_0}^t \left(-  \mu \cos \vartheta \, E(\tau)  + \mu^2 C \tilde C \right) \, e^{\mu \tau \cos \vartheta } \phi(\tau) \, d\tau,
\]
and thus 
\[
\begin{split}
-\int_{t_0}^t E(\tau) e^{\mu \tau \cos \vartheta } \phi'(\tau) \, d \tau &\le
\int_{t_0}^t \mu^2 C \tilde C  e^{\mu \tau \cos \vartheta } \phi(\tau) \, d\tau\\
&\le \int_{t_0}^t \frac{\mu^2 C \tilde C}{\mu \cos \vartheta}(1-  e^{\mu \tau \cos \vartheta }) \phi'(\tau) \, d\tau.
\end{split}
\]
Hence
\[
E(t) e^{\mu t \cos \vartheta } + \frac{\mu^2 C \tilde C}{\mu \cos \vartheta}(1-  e^{\mu t \cos \vartheta })
\le E(t_0) e^{\mu t_0 \cos \vartheta } + \frac{\mu^2 C \tilde C}{\mu \cos \vartheta}(1-  e^{\mu t_0 \cos \vartheta }),
\]
for almost all $0\leq t_0 \leq t$. In summary, for almost all $t\ge 0$ we have
\[
E(t) \le E(0) e^{-\mu t \cos \vartheta } +K(1-  e^{-\mu t \cos \vartheta }),
\]
where
\[
K= \frac{\mu C \tilde C}{\cos \vartheta}.
\]
However, since $\psi\in C([0,\infty;\Sigma))$ implies that $E (t)$ is continuous in time, we 
consequently infer the inequality for all $t\ge 0$.\end{proof} 

\begin{remark}
The proof above is slightly complicated due to the fact that we cannot use the energy identity  \eqref{eq:energy} in its differentiated form \eqref{energy-derivative}, see 
the discussion below the proof of Lemma \ref{lem:Lyapunov}. 
If we ignore this problem for the moment, then we have
\[
\frac{d}{dt} E(t) \le\mu\frac{d}{dt}M(t)\le -  \mu \cos \vartheta E(t) + \mu^2 C \tilde C,
\]
which directly allows us to conclude the assertion proved above.
\end{remark}

In view of Lemma \ref{lem:glob_bound} the bound on $E(t)$ obtained above implies a similar bound on $M(t)$. In particular, 
there is some constant $\rho_M >0$ and a function $t_M(\cdot)$,  such that for all $\psi\in C([0,\infty);\Sigma)$ solutions to \eqref{NLS_diss}, it holds
\[
\| \psi (t, \cdot)\|_{L^2} \le \rho_M, \quad \forall t \ge t_M(M(0)).
\]
Therefore $$\{ \psi \in L^2(\R^d): \|\psi\|_{L^2} \le \rho_M \}\subset L^2(\R^d)$$
is an absorbing ball for trajectories $t \mapsto \psi(t, \cdot)$. Similarly, we know, that there exists a 
$\rho_\Sigma\ge \rho_M$ and a function $t_\Sigma(\cdot)$,  such that
\[
\|\psi (t, \cdot)\|_\Sigma \le \rho_\Sigma, \quad \forall t \ge t_\Sigma(\|\psi(0)\|_\Sigma).
\]
In other words, 
\begin{equation}\label{X}
X:=\{ \psi \in \Sigma: \|\psi\|_{\Sigma} \leq  \rho_\Sigma \}
\end{equation}
is an absorbing ball in $\Sigma$ for trajectories $t\mapsto \psi(t, \cdot)$.
In our study of long time dynamics of \eqref{NLS_diss}, the set $X$ will play the role of a {\it phase space}.

\section{The global attractor and its properties} \label{sec:attractor}
In the previous section we proved that solutions $\psi(t)$ exist globally in $\Sigma$, and, moreover, all such solutions
remain within an absorbing ball $X \subset \Sigma$ for $t>0$ large enough. It is therefore natural to ask whether
there exists an $\A \subset \Sigma$ that {\it attracts} all trajectories $t\mapsto \psi(t, \cdot)\in \Sigma$. Unfortunately,
classical theories of global attractors (see, e.g., \cite{CV, Te}) do not apply to our situation as they typically require asymptotic
compactness, which is unknown in $\Sigma$. However, the trajectories might still converge to the global attractor $\A$
in some weaker metric, say $L^2$. To prove this we revisit the rather general framework of evolutionary systems introduced in \cite{C5} and adapt it to our situation.

\subsection{Existence of a global attractor}

First, recall that our phase space is the metric space $(X,\dr_{L^2}(\cdot,\cdot))$ where $X\subset \Sigma$ is given by \eqref{X} and 
$\dr_{L^2}({\psi,\phi}) = \| {\psi - \phi} \|_{L^2}$. We note that $X$ is $\dw$-compact.
In addition, we also have the stronger $\Sigma$-metric $\ds({\psi,\phi}):= \| {\psi-\phi}\|_{\Sigma}$ on $X$, which satisfies: If $\ds(\psi_n, \phi_n) \to 0$ as $n \to \infty$ for some
$\psi_n, \phi_n \in X$, then $\dw(\psi_n, \phi_n) \to 0$ as $n \to \infty$.
Note that any $\Sigma$-compact set is $L^2$-compact, and any $L^2$-closed set is $\Sigma$-closed.

Now, let $C([a, b];X_\bullet)$, where $\bullet = \Sigma$ or $L^2$, be the space of $\db$-continuous $X$-valued
functions on $[a, b]$ endowed with the metric
\[
\dd_{C([a, b];X_\bullet)}(\psi,\phi) := \sup_{t\in[a,b]}\db(\psi(t),\phi(t)). 
\]
Also, let $C([a, \infty);X_\bullet)$ be the space of $\db$-continuous
$X$-valued functions on $[a, \infty)$ endowed with the metric
\[
\dd_{C([a, \infty);X_\bullet)}(\psi,\phi) := \sum_{T\in \mathbb{N}} \frac{1}{2^T} \, \frac{\sup\{\db(\psi(t),\phi(t)):a\leq t\leq a+T\}}
{1+\sup\{\db(\psi(t),\phi(t)):a\leq t\leq a+T\}}.
\]
In order to define a general evolutionary system, we introduce
\[
\mathcal{T} := \{ I: \ I=[T,\infty) \subset \mathbb{R}, \mbox{ or } 
I=(-\infty, \infty) \},
\]
and for each $I \subset \mathcal{T}$, we denote 
the set of all $X$-valued functions on $I$ by $\mathcal X(I)$.

\begin{definition} \label{Dc}
A map $\Ec$ that associates to each $I\in \mathcal{T}$ a subset
$\Ec(I) \subset \mathcal X(I)$ will be called an {\it evolutionary system} if
the following conditions are satisfied:

\begin{itemize}
\item[(i)] $\Ec([0,\infty)) \ne \emptyset$.
\item[(ii)]
$\Ec(I+s)=\{\psi(\cdot): \ \psi(\cdot +s) \in \Ec(I) \}$ for
all $s \in \mathbb{R}$.

\item[(iii)] For all pairs $I_2\subset I_1 \in \mathcal{T}$: $\{\psi(\cdot)|_{I_2} : \psi(\cdot) \in \Ec(I_1)\}
\subset \Ec(I_2)$.

\item[(iv)]
$\Ec((-\infty , \infty)) = \{\psi(\cdot) : \ \psi(\cdot)|_{[T,\infty)}
\in \Ec([T, \infty)) \ \forall T \in \mathbb{R} \}.$
\end{itemize}
In general, $\Ec(I)$ will be referred to as {\it set of trajectories} on the time interval $I$, and trajectories in $\Ec((-\infty,\infty))$ will be called {\it complete}.
\end{definition}

We now consider the specific evolutionary system induced by the family of trajectories of \eqref{NLS_diss}
in $X$. More precisely, we set
\begin{equation}
\begin{split}
\Ec([T,\infty)) := \Big \{  &\,  \psi \in C([T, \infty);X) \ \text{a solution to \eqref{NLS_diss}, with $\vartheta \in \big (-\frac \pi  2 ,\frac \pi  2 \big)$,} \\ 
&\, \text{$\lambda, \mu > 0$,  $\omega > |\Omega| $, and $0 < \sigma < \frac{d}{2(d-2)}$}  \Big \} .
\end{split}
\label{eq:syst}
\end{equation}
Clearly, the properties (i)--(iv) above hold for the evolutionary system associated to \eqref{NLS_diss}. In addition, due to Proposition \ref{prop:global}, for
any $\psi_0 \in X$ there exists $\psi \in \Ec([{T},\infty))$ with $\psi({T})=\psi_0$. Standard techniques then imply the following lemma:
\begin{lemma} \label{l:convergenceofLH}
Let $(\psi_n)_{n\in \N}$ be a sequence of functions, such that $\psi_n \in \Ec([T_1, \infty))$ for all $n\in \N$. Then for any $T_2>T_1$ there exists a sub-sequence $(\psi_{n_j})_{j\in \N}$ which converges in 
$C([T_1, T_2]; X_{L^2})$ to $\psi\in \Ec([T_1, \infty))$.
\end{lemma}

\begin{proof}
Since $X$ is compact in $L^2(\R^d)$, there exists a sequence $(\psi_{n_j})_{j\in \N}$ such that $\psi_{n_j}(T_1) \to \tilde \psi$ for some $\tilde \psi \in L^2(\R^d)$. 
However, since lower-semicontinuity and the definition of $X$ yield
\[
\| \tilde \psi \|_\Sigma \le \liminf_{j \to \infty} \| \psi_{n_j}(T_1, \cdot) \|_\Sigma  \le \rho_\Sigma,
\]
we have that $\tilde \psi \in X$. In view of proposition \ref{prop:global} there exists $\psi \in \Ec([T_1, \infty))$ with $\psi(T_1) = \tilde \psi$. 
Continuous dependence on the initial data, then gives the desired result.
\end{proof}

Using this, we can prove one of the main structural properties of the set of trajectories induced by \eqref{NLS_diss}:
\begin{proposition} \label{prop:compact} $\Ec([0,\infty))$ is a compact set in $C([0,\infty); X_{L^2})$.
\end{proposition}
\begin{proof}
First note that $\Ec([0,\infty)) \subset C([0,\infty);X_{L^2})$. Now take any sequence
$(\psi_n)_{n\in \N} \in \Ec([0,\infty))$.
Thanks to Lemma~\ref{l:convergenceofLH}, there exists
a subsequence, still denoted by  $\psi_n$, that converges
to some $\psi^{1} \in \Ec([0,\infty))$ in $C([0, 1];X_{L^2})$ as $n \to \infty$.
Passing to a subsequence and dropping a subindex once more, we obtain that
$\psi_n \to \psi^2$ in $C([0, 2];X_{L^2})$ as $n \to \infty$ for some 
$\psi^{2} \in \Ec([0,\infty))$.
Note that $\psi^1(t)=\psi^2(t)$ on $[0, 1]$.
Continuing and picking a diagonal sequence, we obtain a subsequence $\psi_{n_j}$
of $\psi_n$ that converges
to some $\psi \in \Ec([0,\infty))$ in $C([0, \infty);X_{L^2})$ as $n_j \to \infty$.
\end{proof}

In order to proceed further, we denote, as usual, the set of all subsets of $X$ by $P(X)$.
For every $t \ge 0$, we can then define a map $R(t):P(X) \to P(X)$, by
\[
R(t)A := \{\psi(t): \psi(0) \in A, \ \text{such that} \ \psi \in \Ec([0,\infty))\}, \quad
\text{for any $A \subset X.$}
\]
Note that the assumptions on $\Ec$ imply that $R(s)$ enjoys
the following property:
\begin{equation} \label{eq:propR(T)}
R(t+s)A \subset R(t)R(s)A, \qquad A \subset X,\quad t,s \ge 0.
\end{equation}

\begin{definition}
A set $A$ is called {\it invariant} under the dynamics, if $R(t)A = A$ for all $t\ge 0$.
\end{definition}

We also recall the standard notion of and $\omega$-limit associated to an evolutionary system (see also \cite{Te}).

\begin{definition}The $\wb${\it-limit} ($\bullet= \Sigma, \, L^2$) of a set $A\subset X$ is
\[\wb(A):=\bigcap_{T\ge0}\overline{\bigcup_{t\ge T}R(t)A}^{\bullet}.\]\end{definition}

We also note that an equivalent definition of the $\wb$-limit set is given by
\[
\begin{split}
\wb(A)=\big\{&{\psi}\in X: \mbox{ there exist sequences }
t_n \xrightarrow{n\to \infty}\infty \mbox { and } {\psi_n} \in R(t_n)A,\\
& \mbox{such that } {\psi_n(t_n)}\xrightarrow{n\to \infty}\psi \mbox{ in the } \db\mbox{-metric} \big \}.
\end{split}
\]

Finally, we will give a precise definition of what we mean by an attractor.
\begin{definition}
A set $A \subset X$ is a $\mathrm{d}_{\bullet}${\it -attracting set}, if it uniformly
attracts $X$ in $\mathrm{d}_{\bullet}$-metric, i.e. 
\[
\liminf_{{\phi} \in A}\mathrm{d}_{\bullet}(R(t) X, {\phi}) \xrightarrow{t\to +\infty}0.
\]
A set
$\mathcal{A} \subset X$ is a
$\mathrm{d}_{\bullet}$-{\it global attractor}  if
$\mathcal{A}$ is a minimal $\mathrm{d}_{\bullet}$-closed
$\mathrm{d}_{\bullet}$-attracting  set.
\end{definition}

After these preparations, we are able to prove the main result of this section:
\begin{corollary} \label{thm:Attractor}
The evolutionary system \eqref{eq:syst} possesses a unique $d_{L^2}$-global attractor $\A=\omega_{L^2}(X)$, which has the following structure
\[
\A=\{\psi_0: \psi_0 =\psi(0) \mbox { for some } \psi\in \Ec((-\infty,\infty))\}\\
\]
Furthermore, it holds:
\begin{enumerate}
\item For any $\epsilon > 0$ and $T > 0$, there exists a $t_0\in \R$, such that for any $t^* > t_0$, every trajectory
 $\psi\in\Ec([0,\infty))$ satisfies $\dw(\psi(t), \phi(t)) < \varepsilon$, for all $t \in [t^*, t^* +T ]$, where $\phi \in \Ec ((-\infty, \infty))$ is some complete trajectory, i.e., 
 the uniform tracking property holds.
\item If the $\Sigma$ global attractor exists, then it coincides with $\A$.
\item $\A$ is connected in $L^2$.
\item $\mathcal A$ is the maximal invariant set.
\end{enumerate}
\end{corollary}

\begin{proof}
Assertion (1) and (2) follow from the results proved in \cite{C5}. To this end, one first shows that the $\omega_{L^2}$-limit of $X$ is an attracting set, which by definition is closed 
and the minimal set satisfying these two properties. Then, using Proposition \ref{prop:compact} and a diagonalization process, one can prove the structural properties of $\mathcal A$, cf. 
\cite[Theorem 5.6]{C5}.
The fact that $\mathcal A$ is connected then follows from Lemma~\ref{l:convergenceofLH} and uniqueness: We argue by contradiction and hence assume 
that $\A$ is not $L^2$-connected. Then there exist disjoint $\dw$-open sets $U_1, U_2 \in X$
such that $\A \subset U_1 \cup U_2$ and $\A \cap \, U_1$, $\A \cap U_2$ are nonempty. Define
\[
X_j= \{ {\psi}\in X: \omega_{L^2}({\psi}) \in U_j\}, \qquad j=1,2.
\] 
Since $U_1$ and $U_2$ are disjoint, we also have that $X_1$, $X_2$ are disjoint. Continuity of trajectories implies that
$X_1 \cup X_2 =X$. Since $\A$ is $\dw$-attracting,  there exists $T> 0$ such that
\[
R(t)X \in U_1 \cup U_2, \qquad \forall t > T.
\]
By continuity of trajectories we have that for each $\psi \in \Ec([0,\infty))$, either $\psi(t) \in U_1$ for all $t >T$, or $\psi(t) \in U_2$ for all $t >T$. 
This implies that both $X_1$ and $X_2$ are nonempty. Moreover, Lemma~\ref{l:convergenceofLH} implies that $X_1$ and
$X_2$ are $\dw$-open. This contradicts the fact that $X$ is $\dw$-connected. 
Finally we note that the structure of $\mathcal A$, together with uniqueness of solutions, imply that $\mathcal A$ is an invariant set. 
Clearly, only complete trajectories are invariant, hence $\mathcal A$ is the maximal invariant set.
\end{proof}

\begin{remark}
In the case of the usual GL equation (posed on bounded domains $D\subset \R^d$) many more details concerning the global attractor are known, see, e.g., \cite{MM, Te, TW}. 
It is an interesting open problem to check which of these results can be extended to our situation and what the main structural differences between \eqref{NLS_diss} and the usual GL equation are.
\end{remark}

\subsection{Dimension of the attractor} \label{sec:dimension} We hereby follow the, by now, classical theory of estimating the Lyapunov numbers associated to $\Ec([0,\infty))$ by studying 
the evolution of an $m$-dimensional volume element of our phase space $X$, cf. \cite[Chapter V]{Te} for a general introduction. Using this technique, 
the case of the usual GL equation on bounded domains $D\subset \R^n$, 
with $n=1,2$ is studied, e.g., in \cite[Chapter VI, Section 7]{Te}. In our case, the same idea works, but requires several adaptions on a technical level. 

To this end, we first rewrite \eqref{NLS_diss} as
\[
\partial_t \psi = - e^{-i \vartheta } G(\psi),\quad \psi_{\mid t =0} = \psi_0.
\]
and, for any $\psi_0 \in \mathcal A$, consider the linearization around a given orbit $\psi(t) = R(t) \psi_0$, i.e.,
\begin{equation}\label{eq:linear}
\partial_t \phi = - e^{-i \vartheta } G'(\psi) \phi, \quad \psi_{\mid t =0} = \xi.
\end{equation}
Here, $\xi \in X$ and $G'$ denotes the Frechet derivative
\[
G'(\psi) \phi = H_{\Omega} \phi  - \mu \phi + \lambda \left(|{\psi}|^{2\sigma} \phi + \sigma \psi |\psi|^{2\sigma -2} \text{Re}\, (\overline \psi \phi) \right),
\]
where $H_\Omega$ is the linear Hamiltonian (with rotation) defined 
in \eqref{H}. It is easy to see, that the linearized equation \eqref{eq:linear} admits a unique strong solution for any given $\xi \in X$ and $\psi\in \mathcal A$. 
We now consider $\phi_1(t), \dots, \phi_m(t)$ solutions to \eqref{eq:linear}, corresponding to initial data $\xi_1, \dots, \xi_m$, $m\in \N$, and choose an 
$L^2$-orthonormal basis $\chi_1(t), \dots, \chi_m(t)$ of 
\[
P_m(t) X:= \text{span}\{ \phi_1(t), \dots, \phi_m(t)\},
\]
where $P_m$ denotes the corresponding orthogonal projection. Then, it is easy to see (cf. \cite{Te}), that the evolution of the $m$-dimensional volume element  in $X$ is given by
\[
|\phi_1(t)\wedge \dots \wedge \phi_m(t)| = | \xi_1 \wedge \dots \wedge \xi_m| \exp \left(- \int_0^t  \text{Re} \, \text{Tr}\,  e^{-i \vartheta }  G'(\psi(s)) \circ P_m(s) \, ds \right).
\]
In order to proceed, we first note that:

\begin{lemma}\label{lem:collective} Let $H_0$ be given by \eqref{H_0}. Then, for any orthonormal family $\{ \chi_j\}_{j=1}^m \subset L^2(\R^d)$ there exists a constant $c=c(\omega, d)>0$, such that
\[
\sum_{j=1}^m \langle H_0 \chi_j , \chi_j \rangle_{L^2} \ge  c \, m^{1+1/d}.
\]
\end{lemma}

\begin{proof}
Having in mind the form and the degeneracy of the eigenvalues stated Lemma \ref{lem:H}, one checks that when counted with multiplicity $E_{0,m}\sim  m^{1/d}$, as $m\to \infty$.
The desired result then follows directly from \cite[Chapter VI, Lemma 2.1]{Te}. \end{proof}

Using this, we can prove the following result for the dimension of $\mathcal A$:

\begin{proposition}\label{prop:dimension}
Consider the dynamical system \eqref{eq:syst} with $\sigma \ge \frac{2}{d}$, and let $m\in \mathbb{N}$ be defined by
\[
m-1 < \left( \frac{\kappa_2 }{\kappa_1} \right)^{d/(d+1)}\le m,
\]
where
\[
\kappa_1 = \frac{\gamma c}{4}\Big(1- \frac{\Omega^2}{\omega^2}\Big), \quad \kappa_2 =  c' \gamma \mu^{1+d} \Big(1- \frac{\Omega^2}{\omega^2}\Big)^{-d}+
\frac{c''( \lambda |\beta| )^{1+\alpha} }{\gamma^\alpha}\Big(1- \frac{\Omega^2}{\omega^2}\Big)^{-\alpha} \delta,
\]
with $c, c', c'', \alpha, \tilde \alpha$ positive constants depending only on $\omega, d, \sigma$, and
\[
\delta = \limsup_{t\to \infty} \sup_{\psi_0 \in \mathcal A} \left( \frac{1}{t} \int_0^t   \| R(s) \psi_0\|^{2\sigma \tilde \alpha}_{L^{2\sigma +2}}  \, ds \right) \le \left(K\frac{\sigma+1}{\lambda}\right)^{\frac{2\sigma}{2\sigma+2-d\sigma}}.
\]
Here, $K$ is the constant from Proposition~\ref{prop:Ebound}.

Then, as $t\to +\infty$, the $m$-dimensional volume element in $X$ is exponentially decaying.
Moreover, the fractal (and hence Hausdorff) dimension of $\mathcal A$ is less than or equal to $m$.
\end{proposition}

\begin{proof}
Having in mind the representation formula for the $m$-dimensional volume element as given above, we introduce
\[
q_m:= \limsup_{t\to \infty} \sup_{ \| \xi_j \|_{L^2} \le 1} \left(- \frac{1}{t} \int_0^t  \text{Re} \, \text{Tr}\,  e^{-i \vartheta }  G'(\psi(s)) \circ P_m(s) \, ds \right).
\]
and quote the following result from \cite[Chapter III, Corollary 4.2]{CV}: If there are constants $\kappa_{1,2}\ge 0$, such that
\[
q_j \le - \kappa_1 j^{\theta} + \kappa_2, \quad \forall j\geq 1,
\]
then the fractal dimension of $\mathcal A $ enjoys the following bound:
\[
\mathrm{d}_F(\mathcal A)\leq \left( \frac{\kappa_2 }{\kappa_1} \right)^{1/\theta}.
\]
In order to obtain the required estimate on $q_j$, we first note that
\[
\text{Re}\,  \text{Tr} \, e^{-i \vartheta } G'(\psi(t)) \circ P_m(t) = \sum_{j=1}^m \text{Re}\, \langle e^{-i \vartheta } G'(\psi(t)) \chi_j(t), \chi_j(t)\rangle_{L^2}.
\]
Next, we recall that $e^{-i \vartheta } = \gamma + i \beta$, with $\gamma >0$, and compute (suppressing all the $t$-dependence for a moment)
\begin{align*}
&\, - \text{Re}\, \langle e^{-i \vartheta } G'(\psi) \chi_j, \chi_j\rangle_{L^2} =  - \frac{\gamma}{2}\left( \| \nabla \chi_j \|^2_{L^2} + \omega^2 
\| x \chi_j \|_{L^2}^2\right) + \gamma \Omega \int_{\R^d} \overline {\chi_j}L\chi_j \, dx + \gamma \mu \\
&\ - \lambda \gamma \int_{\R^d} |\psi|^{2\sigma} |\chi_j|^2 \, dx + \sigma \lambda  \int_{\R^d}  |\psi|^{2\sigma -2} 
\text{Re}\, (\overline \psi \chi_j)\big( \beta \text{Im}\, (\psi \overline{\chi_j})- \gamma \text{Re}\, (\psi \overline {\chi_j}) \big)\, dx,
\end{align*}
where we have also used the fact that $\| \chi(t) \|_{L^2}=1$. Next, we estimate the term {proportional to} $\Omega$ as we did in the proof of Lemma \ref{lem:Ebound} and we also 
use that fact that 
\[
\int_{\R^d} \beta \text{Re}\, (\overline \psi \chi_j) \beta \text{Im}\, (\psi \overline{\chi_j})\,dx \le |\beta| \int_{\R^d} |\psi|^2 |\chi_j|^2\, dx.
\]
In summary, this yields
\begin{align*}
 - \text{Re}\, \langle e^{-i \vartheta } G'(\psi) \chi_j, \chi_j\rangle_{L^2} \le &\, - \frac{\gamma}{2}\Big(1- \frac{\Omega^2}{\omega^2}\Big)\left( \| \nabla \chi_j \|^2_{L^2} + \omega^2 \| x \chi_j \|_{L^2}^2\right) 
 + \gamma \mu \\
&\ - \lambda (\gamma- \sigma |\beta| ) \int_{\R^d} |\psi|^{2\sigma} |\chi_j|^2 \, dx .
\end{align*}
Thus,
\begin{equation}\label{mest1}
\begin{split}
- \sum_{j=1}^m \text{Re}\, \langle e^{-i \vartheta } G'(\psi) \chi_j, \chi_j\rangle_{L^2} \le &\, - \gamma \Big(1- \frac{\Omega^2}{\omega^2}\Big)\sum_{j=1}^m \langle H_0 \chi_j , \chi_j \rangle_{L^2}   + \gamma \mu m   \\
&\,  + \sigma \lambda |\beta| \sum_{j=1}^m \int_{\R^d} |\psi|^{2\sigma} |\chi_j|^2 \, dx,
\end{split}
\end{equation}
in view of definition \eqref{H_0}. To further estimate the right hand side of \eqref{mest1}, we use H\"older's inequality and Gagliardo-Nirenberg to obtain
\[
\sum_{j=1}^m \int_{\R^d} |\psi|^{2\sigma} |\chi_j|^2 \, dx \le \| \psi\|^{2\sigma}_{L^{2\sigma +2}}
\left(\int_{\mathbb{R}^d}\left( \sum_{j=1}^m  |\chi_j|^2\right)^{\sigma+1}\right)^{\frac{1}{\sigma+1}}.
\]
We will use the generalized Sobolev-Lieb-Thirring inequality (see \cite{GMT}) that reads 
\[
\left(\int_{\mathbb{R}^d} \left( \sum_{j=1}^m |\chi_j|^2\right)^{\frac{p}{p-1}} \, dx
\right)^{\frac{2(p-1)}{d}}
\le c_1 \sum_{j=1}^m \int_{\mathbb{R}^d} |\nabla \chi_j|^2 \, dx,
\]
provided $\max\{1,d/2\} <p\leq 1+d/2$. Here $c_1=c_1(d,p)>0$ some absolute constant. Choosing $p=1+\frac{1}{\sigma}$ (which requires $\sigma\geq 2/d$), we obtain
\[
\left(\int_{\mathbb{R}^d}\left( \sum_{j=1}^m  |\chi_j|^2\right)^{\sigma+1}\right)^{\frac{1}{\sigma+1}} \lesssim
\left(\sum_{j=1}^m \|\nabla \chi_j\|_{L^2}^2\right)^{\frac{d\sigma}{2(\sigma+1)}}.
\]
 Young's inequality then implies that for any $\varepsilon >0$, there exists a $c_2 = c_2(d, \sigma)>0$, such that
\begin{align*}
 \| \psi\|^{2\sigma}_{L^{2\sigma +2}} \left(\sum_{j=1}^m \|\nabla \chi_j\|_{L^2}^2\right)^{\frac{d\sigma}{2(\sigma+1)}} \le &\, \frac{c_2}{\varepsilon^{\alpha}}  \| \psi\|^{2\sigma \tilde \alpha}_{L^{2\sigma +2}} + \varepsilon \sum_{j=1}^m \|\nabla \chi_j\|_{L^2}^2 \\
 \le & \, \frac{c_2}{\varepsilon^{\alpha}}  \| \psi\|^{2\sigma \tilde \alpha}_{L^{2\sigma +2}} + \varepsilon \sum_{j=1}^m\langle H_0 \chi_j , \chi_j \rangle_{L^2},
 \end{align*}
 where $\alpha = \frac{d\sigma}{2\sigma +2-d\sigma}$, and $\tilde \alpha = \frac{2\sigma+2}{2\sigma +2-d\sigma}$. Note that both of these exponents are positive for $\sigma <\frac{d}{2(d-2)}$.
Thus, we an appropriate choice of $\varepsilon$, we obtain from \eqref{mest1}, that
\begin{equation*}\label{mest2}
\begin{split}
- \sum_{j=1}^m \text{Re}\, \langle e^{-i \vartheta } G'(\psi) \chi_j, \chi_j\rangle_{L^2} \le &\, - \frac{\gamma}{2} \Big(1- \frac{\Omega^2}{\omega^2}\Big)\sum_{j=1}^m \langle H_0 \chi_j , \chi_j \rangle_{L^2}   + \gamma \mu m   \\
&\,  + \frac{c_3( \lambda |\beta| )^{1+\alpha} }{\gamma^\alpha}\Big(1- \frac{\Omega^2}{\omega^2}\Big)^{-\alpha}  \| \psi\|^{2\sigma \tilde \alpha}_{L^{2\sigma +2}}.
\end{split}
\end{equation*}
Now, using the estimate from Lemma \ref{lem:collective} above, we have
\begin{equation*}\label{mest3}
\begin{split}
- \sum_{j=1}^m \text{Re}\, \langle e^{-i \vartheta } G'(\psi) \chi_j, \chi_j\rangle_{L^2} \le &\, - \frac{\gamma c}{2} \Big(1- \frac{\Omega^2}{\omega^2}\Big)m^{1+1/d}  + \gamma \mu m   \\
&\,  + \frac{c_3( \lambda |\beta| )^{1+\alpha} }{\gamma^\alpha}\Big(1- \frac{\Omega^2}{\omega^2}\Big)^{-\alpha}  \| \psi\|^{2\sigma \tilde \alpha}_{L^{2\sigma +2}}.
\end{split}
\end{equation*}
This can be estimated further by
\[
- \sum_{j=1}^m \text{Re}\, \langle e^{-i \vartheta } G'(\psi(t)) \chi_j, \chi_j\rangle_{L^2} \le - \kappa_1 m^{1+1/d}  + \rho(t) ,
\]
where $\kappa_1$ is as defined above and
\[
\rho(t) =  c_4 \gamma \mu^{1+d} \Big(1- \frac{\Omega^2}{\omega^2}\Big)^{-d}+
\frac{c_3( \lambda |\beta| )^{1+\alpha} }{\gamma^\alpha}\Big(1- \frac{\Omega^2}{\omega^2}\Big)^{-\alpha}  \| \psi (t)\|^{2\sigma \tilde \alpha}_{L^{2\sigma +2}}, 
\]
with $c_4=c_4(\omega, d)>0$. 

Now, for $\psi(t) = R(t)\psi_0 \in \mathcal A$, we have that
\[
\delta = \limsup_{t\to \infty} \sup_{\psi_0 \in \mathcal A} \left( \frac{1}{t} \int_0^t   \| R(s) \psi_0\|^{2\sigma \tilde \alpha}_{L^{2\sigma +2}}  \, ds \right) <\infty,
\]
due to Lemma \ref{lem:Ebound} and Proposition \ref{prop:Ebound}, which imply that for $\psi(t)\in \mathcal A$:
\[
\| \psi(t) \|^{2\sigma \tilde \alpha}_{L^{2\sigma +2}} \lesssim \| \psi(t) \|^{2\sigma \tilde \alpha}_{\Sigma} \lesssim \rho_\Sigma ^{2\sigma \tilde \alpha}.
\]
This consequently yields
\[
q_m \le - \kappa_1 m^{1+1/d} + \kappa_2, \quad \text{ for all $m\ge 1$,}
\] 
which finishes the proof.
\end{proof}

\begin{remark}
In comparison to many 
classical results on the dimensions of global attractors (e.g., \cite{Te}), the proof above requires the use of the {\it generalized} Lieb-Thirring type inequality to control the 
term proportional to $ \lambda$, see \cite{GMT} for more details. 
\end{remark}

We expect that a similar analysis can be done to estimate the box dimension of the attractor, cf. \cite{CV} for more details.
We finally note that a careful analysis of all the involved constants in $\kappa_1, \kappa_2$ shows that for a given, fixed $\omega>0$, 
the fraction
\[
\left( \frac{ \kappa_2 }{\kappa_1} \right) \to +\infty, \quad \text{as $|\Omega| \to \omega$.}
\]
The estimate on the dimension of $\mathcal A$ thus becomes larger the larger the rotation speed.

\appendix

\section{Derivation of the kernel of the linear semi-group}

Our starting point for justifying the formula~\eqref{SG_diss} for the kernel of the linear semigroup, is the following linear Schr\"odinger equation
\[
 i \partial_t u = H_0 u, \, \quad u_{\mid t =0}  = u_0(x).
\]
where, as before, $H_0 =  \frac{1}{2} (-\Delta + |x|^2)$.
For this equation, Mehler's formula yields an explicit representation of (the kernel of) the associated semi-group \cite{Car09}. 
More precisely,
\[
u(t,x) = \int_{\R^d} S_0(t,x,y) u_0(y) dy, 
\]
where 
\begin{align}\label{SG_LS}
 S_0(t,x,y) = (2 \pi i g_0(t))^{-\frac{d}{2}} \exp\bigg(\frac{i}{g_0(t)}\bigg(\frac{h_0(t)}{2}(x^2+y^2) - x\cdot y\bigg)\bigg),
\end{align}
and 
\[g_0(t) = \frac{\sin(\omega t)}{\omega}, \quad  h_0(t) = \cos(\omega t).\]
It was shown in \cite{AMS}, that a simple change of variables allows to obtain an analogous formula for solutions of
 the linear Schr\"odinger equation with non-vanishing rotation. 
To this end, we write
\[
 x_1 = \cos(\Omega t)  \tilde x_1 + \sin(\Omega t) \tilde x_2,\quad  x_2 = \cos(\Omega t) \tilde x_2 - \sin(\Omega t)
\tilde x_1,
\]
and $x_3$ (if applicable) is left unchanged. Note that this transformation is volume-preserving and hence does not affect the pre-factor $(2 \pi i g_0(t))^{-d/2}$ 
which ensures that 
\[
\iint_ {\R^d\times \R^d} S_0(t,x,y)\, dx \, dy =1 ,\quad \text {for all $t>0$ .}
\]
Defining the new unknown $\tilde u(t,x) = u(t,\tilde x)$ a straightforward calculation shows that $\tilde u$ solves
\begin{equation}\label{LS_rot}
 i \partial_t \tilde u= H_0 \tilde u- \Omega L \tilde u, \quad \tilde u_{\mid t =0} = u_0(x),
\end{equation}
Substituting the new coordinates into Mehler's formula \eqref{SG_LS} yields 
\begin{multline*}
 \tilde S_\Omega(t,x,y) = (2 \pi i g_0(t))^{-\frac{d}{2}} \exp\Bigg(\frac{i}{g_0(t)}\bigg(\frac{h_0(t)}{2}(x^2+y^2) - 
 (\cos(\Omega t) x_1 + \sin(\Omega t) x_2)y_1\\ - (\cos(\Omega t) x_2 - \sin(\Omega t) x_1)y_2 \bigg) \Bigg).
\end{multline*}
Here we drop the term $x_3 y_3$ for notational convenience, since it is unchanged by the change of coordinates.
In order to finally obtain $S_\Omega(t,x,y)$, i.e., the kernel for the dissipative semi-group
$
S_\Omega(t)
$
associated to \eqref{NLS_diss} we replace $t \mapsto -ie^{-i\vartheta} t$ in the above kernel. In other words,
\[
S_\Omega(t,x,y) =  \tilde S_\Omega\left(-ie^{-i\vartheta} t,x,y\right),
\]
which, after some algebra, yields \eqref{SG_diss}. In there, 
the pre-factor in front of the exponent is understood in terms of the principal value of the complex logarithm 
via $(a+ib)^\gamma = e^{\gamma \log(a+i b)}$ and is differentiable for small enough $t>0$. 

\smallskip

Next, in order to study the regularizing properties of $S_\Omega(t)$ for short times, we first note that the phase function $\Phi$ in \eqref{SG_diss} can be 
decomposed into its real and imaginary part, denoted by $\Phi=\Phi_1+i \Phi_2$. 
For Gaussian, i.e.\ heat kernel type, regularity properties of the semi-group $S_\Omega(t)$ for small $t>0$,
we require (at least) an inverse quadratic decay of the real part $\Phi_1$. 
To this end, let
\[a=t\omega\cos\vartheta, \quad b=t\omega\sin\vartheta, \quad \varpi=\frac{\Omega}{\omega}.\]
With this choice of notation we have
\begin{align*}
 \frac{1}{\omega}\Phi_1 = &\, \re \bigg(\frac{\cosh(a+ib)}{\sinh(a+ib)}\bigg) \frac{(x^2+y^2)}{2} - \re \bigg( \frac{\cosh(\varpi (a+i b))}{\sinh(a+ib)} \bigg) (x_1y_1 + x_2y_2)\\ 
&\,  + \im \bigg( \frac{\sinh(\varpi (a+i b))}{\sinh(a+ib)} \bigg) (x_1y_2 - x_2y_1)
\end{align*}
Let us investigate the behavior of the real part $\Phi_1$ of the exponent near $t\approx 0$. Standard trigonometric identities yield
\begin{multline}\label{coeff_x1y1}
 \re \bigg( \frac{\cosh(\varpi (a+i b))}{\sinh(a+ib)} \bigg)\\ = \frac{\cosh(\varpi a) \cos(\varpi b) \sinh(a) \cos(b) + \sinh(\varpi a)
\sin(\varpi b) \cosh(a) \sin(b)}{\sin^2(b) + \sinh^2(a)},
\end{multline}
and
\begin{multline}\label{coeff_x1y2}
 \im \bigg( \frac{\sinh(\varpi (a+i b))}{\sinh(a+ib)} \bigg)\\ = \frac{\cosh(\varpi a) \sin(\varpi b) \sinh(a) \cos(b) - \sinh(\varpi a)
\cos(\varpi b) \cosh(a) \sin(b)}{\sin^2(b) + \sinh^2(a)},
\end{multline}
for all $a,b,\varpi\in\R$. In case $\varpi = 1$, the identity $2\sinh(a) \cosh(a) = \sinh(2a)$ yields
\begin{equation}\label{coeff_x^2+y^2}
 \re \bigg(\frac{\cosh(a+ib)}{\sinh(a+ib)}\bigg) = \frac{\sinh(2a)}{2\sin(b)^2  + 2\sinh(a)^2}.
\end{equation}
Now, we can Taylor expand the expressions \eqref{coeff_x1y1}, \eqref{coeff_x1y2} and \eqref{coeff_x^2+y^2} around $t = 0$. 
Recalling $a=t\omega\cos\vartheta$, $b=-t\omega\sin\vartheta$, and $\varpi=\frac{\Omega}{\omega}$, the denominator of all three terms~\eqref{coeff_x1y1}--\eqref{coeff_x^2+y^2} equals
\[
 \sin^2(b) + \sinh^2(a) = \sin^2(t\omega\sin\vartheta) + \sinh^2(t\omega\cos\vartheta)
\]
Straight-forward expansion yields
\[
 \sin^2(t\omega\sin\vartheta) + \sinh^2(t\omega\cos\vartheta) = \omega^2 t^2 + {\textstyle\frac{1}{3}} \omega^4 t^4 \cos(2\vartheta) + O(t^6),
\]
and thus
\[
 \frac{1}{\sin^2(t\omega\sin\vartheta) + \sinh^2(t\omega\cos\vartheta)} = \frac{1}{\omega^2 t^2} \big( 1 - {\textstyle\frac{1}{3}} \omega^2 t^2
\cos(2\vartheta) \big) + O(t^2).
\]
Here and throughout the constant in the $O(\cdot)$-notation only depend on $\vartheta$, $\omega$, and $\Omega$. The numerator of~\eqref{coeff_x^2+y^2} satisfies
\[
 \sinh(-2t\omega\cos\vartheta) = -2t\omega\cos\vartheta - {\textstyle\frac{4}{3}} t^3 \omega^3 \cos^3\vartheta + O(t^5).
\]
On the other hand, the terms in the numerator of \eqref{coeff_x1y1} satisfy
\begin{align*} &\,\cosh(t \Omega \cos\vartheta)\cos(t \Omega \sin\vartheta)\sinh(t\omega\cos\vartheta)\cos(t\omega\sin\vartheta)\\
 &\, =t\omega\cos\vartheta + {\textstyle\frac{1}{2}} t^3 \omega \cos\vartheta \big( \Omega^2\cos^2\vartheta - (\omega^2+\Omega^2) \sin^2\vartheta
\big) + {\textstyle\frac{1}{6}} t^3 \omega^3 \cos^3\vartheta + O(t^5)
\end{align*}
and 
\begin{align*} &\,  \cosh(t\Omega\cos\vartheta)\cos(t\Omega\sin\vartheta)\sinh(t\omega\cos\vartheta)\cos(t\omega\sin\vartheta)\\
&\, =t^3 \omega \Omega^2 \cos\vartheta \sin^2\vartheta + O(t^5).
\end{align*}
Hence, their sum equals
\[
 t\omega\cos\vartheta + {\textstyle\frac{1}{6}} t^3 \omega^3 \cos\vartheta \big( 3\Omega^2 + 2\omega^2 \cos(2\vartheta) -
\omega^2 \big) + O(t^5).
\]
Likewise we can easily check
\begin{align*}
&\cosh(t\Omega\cos\vartheta)\sin(t\Omega\sin\vartheta)\sinh(t\omega\cos\vartheta)\cos(t\omega\sin\vartheta)\\
 &=t^2\omega\Omega\sin\vartheta\cos\vartheta + O(t^4)\\ &=\sinh(t\Omega\cos\vartheta)\cos(t\Omega\sin\vartheta)\cosh(t\omega\cos\vartheta)\sin(t\omega\sin\vartheta)
\end{align*}
and hence the numerator of~\eqref{coeff_x1y2} vanishes up to fourth order. Collecting all the expansions so far, we obtain
\begin{align*}
 \Phi_1(t,x,y) &= \frac{\omega}{t^2 \omega^2} (1 - {\textstyle\frac{1}{3}} t^2 \omega^2 \cos(2\vartheta))
\bigg({\textstyle\frac{1}{4}}(x^2+y^2)\big(-2t\omega\cos\vartheta - {\textstyle\frac{4}{3}} t^3 \omega^3 \cos^3\vartheta\big)\\
 &\qquad + 2(x_1y_1 + x_2y_2) \big(t\omega\cos\vartheta + {\textstyle\frac{1}{6}} t^3 \omega^3 \cos\vartheta ( 3\Omega^2 + 2\omega^2 \cos(2\vartheta) -
\omega^2 )\big)\\&\qquad\qquad\qquad\qquad\qquad\qquad\qquad  + O((x^2+y^2)t^4) \bigg)\\
 &= -\frac{\cos\vartheta}{2 t} \bigg((x^2+y^2)\big(1 + {\textstyle\frac{1}{3}} t^2 \omega^2\big) +
2(x_1y_1 + x_2y_2) \big(- 1 + {\textstyle\frac{1}{6}} t^2 (\omega^2 - 3\Omega^2)
\big)\\&\qquad\qquad\qquad  + O((x^2 + y^2)t^3) \bigg).
\end{align*}
This expression can be further simplified by collecting coefficients of $x-y$ and $x+y$ to obtain
\begin{multline*}
 \Phi_1(t,x,y) = -\frac{\cos\vartheta}{2 t} \bigg((x-y)^2\big(2 + {\textstyle\frac{1}{6}} (\omega^2+3\Omega^2)t^2\big) +\\
(x+y)^2 {\textstyle\frac{1}{2}} (\omega^2 - \Omega^2) t^2 + O((x^2+y^2)t^3) \bigg).
\end{multline*}
Note that $\cos(\vartheta) > 0$ if $\vartheta \in (-\frac{\pi}{2},\frac{\pi}{2})$ and we assume that $\omega>\Omega\ge0$. The term $O((x^2+y^2)t^3)$ can thus be absorbed in the other coefficients for small enough $t$. In particular, if $t<\delta$ for some small enough $\delta>0$, the real part satisfies \[
\Phi_1(t,x,y) \le -\frac{|x-y|^2}{ct}, \quad c>0.\] 
Hence the semi-group has the same decay as the heat kernel and indeed satisfies~\eqref{kernel_Lp}. Since the exponent $F$ is quadratic in $x$ and $y$, the derivative of 
$S$ w.r.t. $x$ yields only an extra linear factor. In summary, we find that for small $t>0$, the absolute value of the kernel is bounded by
\begin{multline*}
 |S_\Omega(t,x,y)| \le (t + O(t^2))^{-\frac{d}{2}} \exp \Bigg( \frac{\cos\vartheta}{2 t} \bigg((x-y)^2\big(2 + {\textstyle\frac{1}{6}} (\omega^2+3\Omega^2)t^2\big)\\
 + (x+y)^2 {\textstyle\frac{1}{2}} (\omega^2 - \Omega^2) t^2 - x_3 y_3 + O((x^2+y^2)t^3)  \bigg)\Bigg).
\end{multline*}

\end{document}